\documentclass[10pt]{scrartcl}
\setkomafont{disposition}{\bfseries\boldmath}
\advance\hoffset-1in
\advance\voffset-1in
\advance\hoffset2.95mm
\title{Computing Flat-Injective Presentations of Multiparameter Persistence Modules}
\author{Fabian Lenzen}

\usepackage{microtype,amsthm,amssymb,mathtools,mleftright,tikz-cd,nicematrix,bm,multirow,booktabs}
\usepackage[disable]{todonotes}
\usepackage[shortcuts]{extdash}
\usepackage[date=year]{biblatex}

\DeclareSourcemap{
	\maps[datatype=bibtex]{
		\map{
			\step[fieldset=eventtitle,null]
			\step[fieldsource=pubstate,match=prepublished,fieldset=pubstate,null]
			\step[fieldset=eprintclass,null]
			\step[fieldset=urldate,null]
		}
		\map{
			\step[fieldsource=doi,final]
			\step[fieldset=isbn,null]
			\step[fieldset=url,null]
			\step[fieldset=issn,null]
			\step[fieldset=eprint,null]
		}
		\map{
			\step[fieldsource=eprint,final]
			\step[fieldset=url,null]
		}
	}
}
\usepackage[format=plain,margin=2em,font=small,labelfont=bf]{caption}
\usepackage[subrefformat=parens]{subcaption}
\usepackage[noend,nosemicolon]{algorithm2e}
\usepackage[hidelinks]{hyperref}
\usepackage{zref-clever}
\usepackage{mysansmath}
\usepackage[patch-newtheorem]{mytheorem-enum}
\zcsetup{cap}
\zcRefTypeSetup{algorithm}{Name-sg=Algorithm, name-sg=algorithm, Name-pl=Algorithms, name-pl=algorithms}
\zcRefTypeSetup{equation}{Name-sg=, name-sg=, Name-pl=, name-pl=}
\zcsetup{countertype={algocf = algorithm}}

\usepackage{xstring,tikz}
\usetikzlibrary{shapes, shapes.misc, backgrounds}

\tikzset{%
	on layer/.code={%
		\pgfonlayer{#1}\begingroup
		\aftergroup\endpgfonlayer
		\aftergroup\endgroup
}}

\newlength{\ModuleDiagramSymbolSize}
\tikzset{%
	really densely dotted/.style={line cap=round, dash pattern=on 0pt off 2\pgflinewidth},
	module diagram symbol size/.style={%
		generator/.style={shape=circle, fill, outer sep=0, inner sep=#1},
		coordinate/.style={shape=circle, draw, outer sep=0, inner sep=#1},
		relation/.style={shape=diamond, fill, outer sep=0, inner sep=#1},
		syzygy/.style={shape=cross out, draw, outer sep=0, inner sep=#1}
	},
	every module diagram/.style={},
	every module/.style={line width=0.5pt},
	module diagram/.style={%
		x=3mm,
		y=3mm,
		module diagram symbol size=.2ex,
		line cap=round,
		every label/.append style={%
			font=\scriptsize,
			label distance=1pt,
			inner sep=0pt,
		},
		every module diagram
	},
	xdotted/.style={dash pattern=on 0pt off 1.5pt},
	diagram fill/.style={fill opacity=0.2},
	diagram transform/.style={shift={(-.5,-.5)}},
	invisible/.style={opacity=0, diagram fill/.style={opacity=0}}
}

\ProvideDocumentCommand{\Grid}{O{0} D(){0.5} m O{0} D(){0.5} m}{%
	\begin{scope}[on background layer]
		\foreach \i in {#1, #2, ..., #3}{
			\foreach \j in {#4, #5, ..., #6}{
				\fill[gray] (\i, \j) circle (0.5pt);
			}
		}
	\end{scope}
}

\NewDocumentCommand{\Axes}{O{->, shorten >=-0pt} r() O{0,0} r()}{%
	\coordinate (ninf) at (#2);
	\coordinate (O) at (#3);
	\coordinate (inf) at (#4);
	\draw[#1] (ninf -| O) -- (O |- inf);
	\draw[#1] (ninf |- O) -- (O -| inf);
	\coordinate (inf) at (#4);
}

\NewDocumentCommand{\ExtendedAxes}{r() O{0,0} m m r()}{%
	\coordinate (ninf) at (#1);
	\coordinate (O) at (#2);
	\coordinate (inf) at (#5);
	\coordinate (startdots) at (#3);
	\coordinate (enddots) at (#4);
	\draw (ninf -| O) -- (O |- startdots);
	\draw[xdotted] (O |- startdots) -- (O |- enddots);
	\draw[->, shorten >=-4pt] (O |- enddots) -- (O |- inf);
	\draw (ninf |- O) -- (O -| startdots);
	\draw[xdotted] (O -| startdots) -- (O -| enddots);
	\draw[->, shorten >=-4pt] (O -| enddots) -- (O -| inf);
}

\NewDocumentCommand{\FreeModule}{D<>{} O{} r() O{} m}{%
	\begin{scope}[#2]
		\node[#4] at (#3) {#5};
		\coordinate (n) at (#3);
		\coordinate (n) at ([diagram transform]n);
		\IfSubStr{#1}{x}{\coordinate (n) at (ninf |- n);}{}
		\IfSubStr{#1}{y}{\coordinate (n) at (ninf -| n);}{}
		\begin{scope}[on background layer=#2]
			\fill[diagram fill] (n) rectangle (inf);
		\end{scope}
		\IfSubStr{#1}{x}{}{\draw (n) -- (n |- inf);}
		\IfSubStr{#1}{y}{}{\draw (n) -- (n -| inf);}
	\end{scope}
}

\NewDocumentCommand{\xyExtendedFreeModule}{O{} r() O{} m}{%
	\begin{scope}[every module, #1]
		\path (#2) node[#3] {#4} coordinate (n);
		\begin{scope}[on background layer=#1]
			\fill[diagram fill] (#2) rectangle (inf);
		\end{scope}
		\draw
		(n) -- (n |- startdots)
		(n |- enddots) -- (n |- inf)
		(n) -- (n -| startdots)
		(n -| enddots) -- (n -| inf);
		\draw[xdotted]
		(n |- startdots) -- (n |- enddots)
		(n -| startdots) -- (n -| enddots);
	\end{scope}
}

\NewDocumentCommand{\xExtendedFreeModule}{O{} r() O{} m}{%
	\begin{scope}[#1]
		\path (#2) node[#3] {#4} coordinate (n);
		\begin{scope}[on background layer=#1]
			\fill[diagram fill] (#2) rectangle (inf);
		\end{scope}
		\draw
		(n |- inf) -- (n) -- (n -| startdots)
		(n -| enddots) -- (n -| inf);
		\draw[xdotted]
		(n -| startdots) -- (n -| enddots);
	\end{scope}
}

\NewDocumentCommand{\yExtendedFreeModule}{O{} r() O{} m}{%
	\begin{scope}[#1]
		\path (#2) node[#3] {#4} coordinate (n);
		\begin{scope}[on background layer=#1]
			\fill[diagram fill] (#2) rectangle (inf);
		\end{scope}
		\draw
		(n -| inf) -- (n) -- (n |- startdots)
		(n |- enddots) -- (n |- inf);
		\draw[xdotted]
		(n |- startdots) -- (n |- enddots);
	\end{scope}
}

\NewDocumentCommand{\FreeComplement}{O{} r() O{} m}{%
	\begin{scope}[#1]
		\path (#2) node[#3] {#4} coordinate (n);
		\coordinate (n) at ([shift={(-.5,-.5)}]n);
		\begin{scope}[on background layer=#1]
			\fill[diagram fill] (n |- inf) |- (n -| inf) |- (ninf) |- (inf -| n);
		\end{scope}
		\draw (n |- inf) |- (n -| inf);
	\end{scope}
}

\NewDocumentCommand{\InjModule}{D<>{} O{} r() O{} m}{%
	\begin{scope}[every module, #2]
		\node[#4] at (#3) {#5};
		\coordinate (n) at (#3);
		\coordinate (n) at ([shift={(.5,.5)}]n);
		\IfSubStr{#1}{x}{\coordinate (n) at (inf |- n);}{}
		\IfSubStr{#1}{y}{\coordinate (n) at (inf -| n);}{}
		\begin{scope}[on background layer=#2]
			\fill[diagram fill] (n) rectangle (ninf);
		\end{scope}
		\IfSubStr{#1}{x}{}{\draw (n) -- (n |- ninf);}
		\IfSubStr{#1}{y}{}{\draw (n) -- (n -| ninf);}
	\end{scope}
}

\NewDocumentCommand{\InjComplement}{O{} r() O{} m}{%
	\begin{scope}[every module, #1]
		\path (#2) node[#3] {#4} coordinate (n);
		\coordinate (n) at ([shift={(-.5,-.5)}]n);
		\begin{scope}[on background layer=#1]
			\fill[diagram fill] (ninf |- n) -| (n |-  ninf) -| (inf) -| cycle;
		\end{scope}
		\draw (ninf |- n) -| (n |-  ninf);
	\end{scope}
}

\NewDocumentCommand{\TruncatedInjModule}{s O{} r() O{} m}{%
	\begin{scope}[every module, #2]
		\coordinate (n) at (#3);
		\IfBooleanTF{#1}{
			\path
			(ninf) node[generator] {}
			(ninf |- n) node[relation] {}
			(ninf -| n) node[relation] {}
			(n) node[syzygy, #4] {#5};
		}{
			\node[#4, at=(n)] {#5};
		}
		\begin{scope}[on background layer=#2]
			\fill[diagram fill] (#3) rectangle (ninf);
		\end{scope}
		\draw
		(n |- enddots) |- (n -| enddots)
		(n |- startdots) |- (enddots |- ninf)
		(startdots |- ninf) -| (startdots -| ninf)
		(enddots -| ninf) |- (startdots |- n);
		\draw[xdotted]
		(n -| enddots) -- (n -| startdots)
		(startdots -| n) -- (enddots -| n)
		(enddots |- ninf) -- (startdots |- ninf)
		(startdots -| ninf) -- (enddots -| ninf);
	\end{scope}
}

\NewDocumentCommand{\TruncatedInjComplement}{s O{} r() O{} m}{%
	\begin{scope}[#2]
		\coordinate (n) at (#3);
		\IfBooleanTF{#1}{
			\path
			(ninf |- n) node[generator] {}
			(ninf -| n) node[generator] {}
			(n) node[relation, #4] {#5};
		}{
			\node[#4, at=(n)] {#5};
		}
		\begin{scope}[on background layer=#2]
			\fill[diagram fill] (n) -- (ninf |- n) -- (ninf |- inf) -- (inf) -- (ninf -| inf) -- (n |- ninf) -- cycle;
		\end{scope}
		\draw
		(inf -| ninf) -- (ninf |- n) -- (startdots |- n)
		(enddots |- n) -| (n |- enddots)
		(n |- startdots) -- (n |-  ninf) -- (ninf -| inf);
		\draw[xdotted]
		(startdots -| n) -- (enddots -| n)
		(startdots |- n) -- (enddots |- n);
	\end{scope}
}

\NewDocumentCommand\StepModule{O{}r()mr()}{%
	\draw[#1] (inf -| #2) -- (#2) #3 -| (#4) -- (#4 -| inf);
	\begin{scope}[on background layer=#1]
		\fill[diagram fill] (inf -| #2) -- (#2) #3 -| (#4) -- (#4 -| inf) |- cycle;
	\end{scope}
}

\makeatletter
\patchcmd\caption@subtypehook{\let\label\subcaption@label}
{\let\label\subcaption@label\let\ltx@label\subcaption@label}{}{\fail}

\newcommand{\theoremfont}{\bfseries}
\newcommand{\theoremlabelfont}{\normalfont}
\newtheoremstyle{plain}{3pt}{3pt}{\itshape}{}{\theoremfont}{.}{.5em}{\thmname{#1}\thmnumber{ #2}\thmnote{ {\theoremlabelfont(#3)}}}
\newtheoremstyle{remark}{3pt}{3pt}{}{}{\theoremfont}{.}{.5em}{\thmname{#1}\thmnumber{ #2}\thmnote{ {\theoremlabelfont(#3)}}}
\newtheoremstyle{definition}{3pt}{3pt}{}{}{\theoremfont}{.}{.5em}{\thmname{#1}\thmnumber{ #2}\thmnote{ {\theoremlabelfont(#3)}}}
\theoremstyle{plain}
\newtheorem{theorem}{Theorem}[section]
\newtheorem{proposition}[theorem]{Proposition}
\newtheorem{lemma}[theorem]{Lemma}
\newtheorem{corollary}[theorem]{Corollary}
\theoremstyle{remark}
\newtheorem*{claim*}{Claim}
\newtheorem{remark}[theorem]{Remark}
\newtheorem{example}[theorem]{Example}
\theoremstyle{definition}
\newtheorem{definition}[theorem]{Definition}
\newenvironment{claimproof}{\begin{proof}[Proof of claim]}{\end{proof}}

\numberwithin{equation}{section}

\let\into\hookrightarrow
\let\onto\twoheadrightarrow
\let\xto\xrightarrow
\let\xinto\xhookrightarrow

\newcommand{\xtofrom}[2][]{
	\mathrel{\mathop{%
		\vcenter{\offinterlineskip\m@th
			\ialign{\hfil##\hfil\cr
				\hphantom{$\scriptstyle\mspace{8mu}{#1}\mspace{8mu}$}\cr
				\rightarrowfill\cr
				\vrule height0pt width 2em\cr
				\leftarrowfill\cr
				\hphantom{$\scriptstyle\mspace{8mu}{#2}\mspace{8mu}$}\cr
				\noalign{\kern-0.3ex}
			}%
		}%
	}\limits^{#1}_{#2}}%
}
\DeclareMathOperator{\coker}{coker}
\DeclareMathOperator{\rk}{rk}
\DeclareMathOperator{\im}{im}
\let\lim\relax
\DeclareMathOperator{\lim}{lim}
\DeclareMathOperator{\Lim}{Lim}
\DeclareMathOperator{\colim}{colim}
\DeclareMathOperator{\Colim}{Colim}
\DeclareMathOperator{\barc}{barc}
\DeclareMathOperator{\piv}{piv}
\DeclareMathOperator{\rg}{rg}
\DeclareMathOperator{\cg}{cg}

\DeclareMathOperator{\Hom}{Hom}
\DeclareMathOperator{\IHom}{\mathsf{Hom}}
\DeclareMathOperator{\Int}{Int}
\mathchardef\mhyphen="2D
\newcommand{\Proj}[1]{\mathcal{F}_{#1}}
\newcommand{\Inj}[1]{\mathcal{I}_{#1}}
\newcommand{\Pers}[1]{#1\mhyphen\mathbf{Pers}}

\newcommand{\Vect}{\mathbf{Vec}}

\newcommand{\Z}{\mathbf{Z}}
\newcommand\lZ{\underline{\Z}}
\newcommand\uZ{\overline{\Z}}

\newcommand{\N}{\mathbf{N}}
\newcommand{\F}{\Bbbk}
\newcommand{\SimpleModule}{\mathbf{k}}
\DeclarePairedDelimiter{\IntSet}{[}{]}
\DeclareMathOperator{\id}{id}
\newcommand{\one}{\bm{1}}
\newcommand{\calO}{\mathcal{O}}
\DeclarePairedDelimiter{\Shift}{\langle}{\rangle}
\DeclarePairedDelimiter{\HShift}{[}{]}
\DeclarePairedDelimiter{\abs}{\lvert}{\rvert}

\DeclarePairedDelimiterX{\Set}[1]{\{}{\}}{\setargs{#1}}
\DeclarePairedDelimiterX{\Span}[1]{\langle}{\rangle}{\setargs{#1}}
\NewDocumentCommand{\setargs}{>{\SplitArgument{1}{;}}m}{\setargsaux#1}
\NewDocumentCommand{\setargsaux}{mm}{\IfNoValueTF{#2}{#1} {#1\nonscript\:\delimsize\vert\allowbreak\nonscript\:\mathopen{}#2}}%
\NewDocumentCommand{\Mtx}{s}{\IfBooleanTF{#1}{\@MtxStar}{\@Mtx}} 
\newenvironment{smallcases}{\left\{\begin{smallmatrix*}[l]}{\end{smallmatrix*}\right.}
\newenvironment{smallarray}[1]{
	\null\,\vcenter\bgroup\scriptsize
	\arraycolsep=.13885em
	\hbox\bgroup$\array{@{}#1@{}}%
}{%
	\endarray$\egroup\egroup\,\null%
}
\NewDocumentCommand{\@Mtx}{om}{%
	\IfNoValueTF{#1}{%
		\mleft(\begin{smallmatrix}#2\end{smallmatrix}\mright)
	}{%
		\mathopen#1(\begin{smallmatrix}#2\end{smallmatrix} \mathclose#1)
	}
}
\NewDocumentCommand{\@MtxStar}{O{c}om}{%
	\IfNoValueTF{#2}{%
		\mleft(\begin{smallmatrix*}[#1]#3\end{smallmatrix*}\mright)
	}{%
		\mathopen#2(\begin{smallmatrix*}[#1]#3\end{smallmatrix*} \mathclose#2)
	}
}
\DeclareRobustCommand{\svdots}{
	\vbox{%
		\baselineskip=0.33333\normalbaselineskip
		\lineskiplimit=0pt
		\hbox{.}\hbox{.}\hbox{.}%
		\kern-0.2\baselineskip
	}%
}
\let\vdots\svdots

\renewcommand{\u}[1]{\mathsf{u}(#1)}

\newlength{\ScratchLen}

\tikzset{every module diagram/.append style={baseline={([yshift=-0.7ex]current bounding box.center)}}}

\SetKw{Break}{break} 
\SetKw{Yield}{yield}
\SetKw{Return}{return}
\SetKwFor{Forever}{forever}{do}{end}
\SetKwProg{Fn}{function}{:}{end}
\SetKwProg{Def}{def}{:}{end}

\SetNlSty{textsf}{}{}
\SetArgSty{sfit}
\SetCommentSty{sfit}
\SetKwSty{sfbf}

\NewDocumentCommand{\FunctionName}{om}{{%
	\textup{\texttt{\IfNoValueTF{#1}{#2}{\hyperref[#1]{#2}}}}%
	\let\tmp\relax%
}}
\NewDocumentCommand{\FunctionArg}{om}{\ensuremath{\IfNoValueTF{#1}{(#2)}{\mathopen#1(#2\mathclose#1)}}}
\NewDocumentCommand{\comment}{s O{f} m}{\tcp*[#2]{\IfBooleanTF{#1}{\smash}{}{\parbox[t]{\commentwidth}{\raggedright #3}}}}
\newlength{\commentwidth}
\usetikzlibrary{calc}
\tikzset{
	math node/.style={execute at begin node=$#1, execute at end node=$},
	math node/.default=\textstyle,
	caption/.style={below=0pt of current bounding box, math node},
}
\makeatother
\nocite{FlangePresentations.jl}
\begin{document}
\maketitle
\begin{abstract}
	A \emph{flat-injective presentation} of a multiparameter persistence module $M$ characterizes $M$ as the image of a morphism from a flat to an injective persistence module.
	Like flat or injective presentations, flat-injective presentations can be easily represented by a single \emph{graded matrix},
	completely describe the persistence module up to isomorphism, and can be used as starting point to compute other invariants of it,
	such as the rank invariant, persistence images, and others.
	
	If all homology modules of a bounded chain complex $F_\bullet$ of flat $n$-parameter modules are finite dimensional,
	it is known that $F_\bullet$ and its shifted image $\nu F_\bullet[n]$ under the Nakayama functor are quasi-isomorphic,
	where $\nu F_\bullet[n]$ is a complex of injective modules.
	We give an explicit construction of a quasi-isomorphism $\phi_\bullet\colon F_\bullet \to \nu F_\bullet[n]$,
	based on the boundary morphisms of $F_\bullet$.

	If $F_\bullet$ is a flat resolution of a finite dimensional persistence module $M$,
	then the degree-zero part $\phi_0\colon F_0 \to \nu F_n$ is a flat-injective resolution of $M$.
	From our construction of $\phi$, we obtain a method to compute a matrix representing $\phi_0$
	from the matrices representing the resolution $F_\bullet$.
	A Julia package implementing this method is available.
\end{abstract}

\section{Introduction}
\paragraph*{Motivation}
Let $\F$ be an arbitrary field,
and $M$ be a one-parameter persistence module,
i.e., a $\Z$-indexed system $\dotsb \to M_{-1} \xto{M_{0,-1}} M_0 \xto{M_{1,0}} M_1 \to \dotsb$ of $\F$-vector spaces and $\F$-linear morphisms.
If all $M_i$ are finite dimensional, then there is a unique (up to isomorphism) decomposition
\begin{equation}
	\label{eq:barcode-decomp}
	M \cong \bigoplus_{i \in I} \Int(b_i, d_i)
\end{equation}
of $M$ for some indexing set $I$ and $-\infty \leq b_i < d_i \leq \infty$,
where $\Int(b, d)$ denotes the \emph{interval module} with
\begin{align*}
	\Int(b, d)_z     &= \begin{smallcases} \F     & \text{if $b \leq z      < d$,} \\ 0 & \text{otherwise},\end{smallcases} &
	\Int(b, d)_{z'z} &= \begin{smallcases} \id_\F & \text{if $b \leq z < z' < d$,} \\ 0 & \text{otherwise}.\end{smallcases}
\end{align*}
The multiset $\barc M \coloneqq \Set{(b_i, d_i); i \in I}$ is uniquely determined by $M$ and called the \emph{barcode} of $M$.

In the multiparameter case, however, i.e., for a commutative system $M$ of vector spaces indexed by $\Z^n$ for $n > 1$,
it is not possible in general to decompose $M$ into interval modules as in \eqref{eq:barcode-decomp};
in fact, the indecomposable modules can be arbitrarily complicated \cite{BuchetEscolar:2018,BauerScoccola:2025}.
Instead, one may represent $M$ in terms of a free or flat presentation; i.e., a morphism $p\colon F_1 \to F_0$ of free (or flat) persistence modules, such that $\coker p \cong M$.
Dually, an injective presentation of $M$ is a morphism $i\colon I^0 \to I^1$ of injective persistence modules such that $\ker i \cong M$ \cite[§5.4]{Miller:2020a}.
Lastly, a \emph{flat-injective presentation} of $M$, a concept introduced by Miller~\cite{Miller:2020a},
is a morphism $\phi\colon F \to I$ from a flat persistence module $F$ to an injective persistence module $I$, such that $M \cong \im \phi$.
Finite flat, injective and flat-injective presentations can be represented by \emph{graded matrices} (also called \emph{monomial matrices}, 
see \hyperref[sec:graded-matrices]{below} and \cites[§2.5]{LesnickWright:2022}[Def.~5.4]{Miller:2020a}).

In the one-parameter case, the barcode decomposition \eqref{eq:barcode-decomp} can be seen as a special flat presentation
\begin{gather*}
	\SwapAboveDisplaySkip
	M \cong \coker\bigl(\smashoperator{\bigoplus_{d_i   <  \infty} } F(d_i)   \xinto{\Phi} \smashoperator{\bigoplus_{-\infty \leq b_i}} F(b_i)\bigr),    \\\shortintertext{injective presentation}
	M \cong   \ker\bigl(\smashoperator{\bigoplus_{d_i \leq \infty} } I(d_i-1) \xinto{\Phi} \smashoperator{\bigoplus_{-\infty   <  b_i-1}} I(b_i-1)\bigr),\\\shortintertext{or flat-injective presentation}
	M \cong    \im\bigl(\smashoperator{\bigoplus_{-\infty \leq b_i}} F(b_i)     \xto{\Phi} \smashoperator{\bigoplus_{d_i \leq \infty} } I(d_i-1)\bigr),
\end{gather*}
where $F(t) \coloneqq \Int(t,\infty)$ is a flat module and $I(t) \coloneqq \Int(-\infty,t)$ is an injective module.
The three presentations are special in the sense that in all three cases, the presentation matrix $\Phi$ has entries $\Phi_{ij} = \delta_{ij}$.
For more than one parameter, such a special (flat, injective or flat-injective) presentation need not exist;
still, general flat, injective and flat-injective presentations can be defined and computed for any number of parameters.

In a sense, a flat-injective presentation sits ``in the middle'' between flat and injective presentations, in the following way.
If $p\colon F_1 \to F_0$ is a flat presentation of of a persistence module $M$ with augmentation morphism $\varepsilon\colon F_0 \onto M$
and $i\colon I^0 \to I^1$ is an injective presentation of $M$ with augmentation morphism $\eta\colon M \into I^0$,
then the composition $\eta\varepsilon\colon F_0 \to I^0$ is a flat-injective presentation of $M$.

Further, recall that the \emph{pointwise} or \emph{Matlis dual} of a persistence module $M$
is the persistence module $M^*$, where $(M^*)_z = (M_{-z})^*$ is the dual vector space of $M_{-z}$.
Since the Matlis dual of a flat module is injective \cite{MillerSturmfels:2005},
if $p$ is a flat presentation of $M$, then $p^*$ is an injective presentation of $M^*$.
In contrast, flat-injective presentations are a self-dual notion, in the sense that if $\phi$ is a flat-injective presentation of $M$, then $\phi^*$ is a flat-injective presentation of $M^*$.

Lastly, each of the three presentation types allows to easily determine the ranks of the structure morphisms $M_{z' z}\colon M_z \to M_{z'}$ of $M$ for $z \leq z'$.
If $U$, $V$ and $W$ are graded matrices that represent a flat, flat-injective and injective presentation of a persistence module $M$, respectively,
then one obtains $\rk M_{z'z}$ by computing the corank, rank and nullity of a certain submatrix $U'$, $V'$ and $W'$, of $U$, $V$ and $W$; see \zcref{rmk:rank-invariant} for details.
These matrices have
\begin{center}
	\begin{tabular}{rccc}
		\toprule
		                &     $U'$      &     $V'$      &     $W'$      \\
		\midrule
		   many rows if & $z$ is large  & $z'$ is small & $z'$ is small \\
		many columns if & $z'$ is large & $z$ is large  & $z$ is small. \\
		\bottomrule
	\end{tabular}
\end{center}
Therefore, the effort for computing $\rk M_{z'z}$ from the different presentation matrices might vary,
depending on the presentation type.

Thus, although flat, flat-injective and injective presentations of a persistence module $M$ are conceptually similar,
contain the same information and can all be seen as multiparameter generalization of the barcode of $M$,
the usefulness for certain tasks may vary in practice.
Nevertheless, research on multiparameter persistence has mostly focused on flat (and free) presentations of persistence modules.
We can only guess reasons for this:
First, expository textbooks on commutative algebra put a stronger emphasis on free, projective and flat presentations and resolutions than on injective ones.
For the second reason, assume that $M$ is given in terms of a \emph{free implicit representation},
which means that $M = \ker f/\im f'$ for a sequence $F' \xto{f'} F \xto{f} F''$ of free persistence modules with $f f' = 0$,
which is the relevant setting in the context of computing multiparameter persistent homology.
For the case of two-parameter persistence, there exist cubic time algorithms \cite{LesnickWright:2019,KerberRolle:2021,companion-computation}
to compute a minimal free presentation or resolution of $M$ from this data.
Also, per the duality result from \cite[Corollary~1.4]{companion-dualities}, every algorithm that computes a free resolution of $M$ from this data
readily yields an injective resolution of $M$.
On the other hand, there exists only since very recently a software package that is able to compute flat-injective presentations,
which is an implementation of the algorithm of \citeauthor{HelmMiller:2005}~\cite[Algorithm~3.6]{HelmMiller:2005}
in the computer algebra system OSCAR \cite{DeckerEderEtAl:2025},
building on top of a highly nontrivial amount of computational commutative algebra.
To our knowledge, this is the only implementation of \cite[Algorithm~3.6]{HelmMiller:2005}.

\paragraph*{Contributions}
In this paper, we close this gap with the following contributions.
Let $F_\bullet$ be a chain complex of free $n$-parameter persistence modules, such that $H_q(F_\bullet)$ has finite total dimension as vector space.
In an earlier paper \cite{companion-dualities}, we showed that the complexes $F_\bullet$ and $\nu F_\bullet[n]$ are quasi-isomorphic, where $\nu$ denotes the Nakayama functor 
(see \zcref{thm:local-duality}).
Our central result, \zcref{thm:main}, gives an explicit construction for a quasi-isomorphism $\phi_\bullet\colon F_\bullet \to \nu F_\bullet[n]$.
In particular, if $F_\bullet$ is the free resolution of a finite dimensional persistence module $M$, then the degree-zero part $\phi_0$ of $\phi_\bullet$
is a flat-injective presentation for $M$; see \zcref{thm:flat-injective presentation module}.
If matrices representing the complex $F_\bullet$ are given, then one obtains a straightforward algorithm to compute matrices representing $\phi_\bullet$.
For the case of $F_\bullet$ being a free resolution, this is carried out in detail in \zcref{sec:flange-pres-matrix}.
Putting things together, we obtain a cubic time algorithm to compute a matrix representing $\phi_0$ from the matrices representing $F_\bullet$,
see \zcref{thm:complexity}.
A self-contained implementation in the Julia programming language is provided in the package \texttt{FlangePresentations.jl}~\cite{FlangePresentations.jl};
see \zcref{sec:implementation}.
Using our implementation, it is possible to compute flat-injective presentations of multiparameter persistent homology modules of typical size.

\paragraph*{Related work}
Computing free presentations (or, more generally, resolutions) has been extensively dealt with in computational commutative algebra.
Typical implementations build upon Schreyer's Algorithm \cite{ErocalMotsakEtAl:2016a}.
In the context of persistent homology, considerable improvements have been achieved for special cases \cite{LesnickWright:2022,FugacciKerberEtAl:2023,Lenzen:2023}.
Furthermore, for multiparameter persistence modules of finite total dimension, there exists a simple correspondence between free and injective resolutions \cite{companion-dualities}
that arises from Grothendieck local duality \cites[278]{Hartshorne:1966}[Theorem~3.5.8]{BrunsHerzog:1998}[Theorem~11.2.6]{BrodmannSharp:2012}
or Greenlees--May duality \cite{GreenleesMay:1992,Faridian:2019,Miller:2002}.

An algorithm for computing flat-injective and injective presentations is provided in \cite[Algorithm~3.6]{HelmMiller:2005},
which has been implemented in the computer algebra package OSCAR \cite{DeckerEderEtAl:2025} recently.
The algorithm expects a description of an input module in terms of generators and relations.
Due to its relying on Gröbner base calculations, it is not expected to be of polynomial complexity.

Another appropach for constructing flat-injective presentations is followed by Stefanou and Grimpen \cite{GrimpenStefanou:2025}.
Given a module $M$, the authors construct a flat-injective presentation of $M$
from the graded components and the transition maps between them. 
If $M$ is finite dimensional, then the constructed presentation has $\sum_{z\in\Z^n} \dim_{M_z}$ generators and cogenerators.
The authors then provide an algorithm for converting a flat-injective presentation into a minimal one.

\paragraph*{Acknowledgements}
We thank Ezra Miller, Anastasios Stefanou and Fritz Grimpen for providing extensive feedback for earlier versions of this paper.
We thank Ulrich Bauer, a question of whom initiated our working on this subject.

\section{Background}
\label{sec:background}
Let $\F$ be an arbitrary field, and let $\Vect$ denote the category of $\F$-vector spaces.

\paragraph*{Persistence modules}
A \emph{$\Z^n$-persistence module} is a functor $M\colon z \mapsto M_z$, $(z \leq z') \mapsto M_{z'z}$ from the poset $\Z^n$ to the category of finite dimensional $\F^n$-vector spaces.
The vector spaces $M_z$ are the \emph{components} and the morphisms $M_{z'z}\colon M_z \to M_{z'}$ the \emph{structure morphisms} of $M$.
A morphism of persistence modules is a natural transformation of functors.
Persistence modules over $\Z^n$ and their morphisms form an abelian category, denoted by $\Pers{n}$, in which limits and colimits (e.g., direct sums, kernels, cokernels or images)
are computed componentwise.
Since a persistence module $M$ is still a $\F$-vector space, it makes sense to consider its (total) dimension $\dim M \coloneqq \sum_z \dim M_z$.
Note that $\Pers{n}$ is equivalent to the category of graded modules over the the $\Z^n$-graded ring $\F[x_1,\dotsc,x_n]$ with finite dimensional graded components \cite{CarlssonZomorodian:2009}.
For the rest of the paper, we will refer to persistence modules simply as \emph{modules}.

\begin{figure}
	\tikzcdset{every diagram/.append style={cramped,font=\scriptsize,row sep={2em,between origins},column sep={2.5em,between origins}, baseline=(B)}}
	\subcaptionbox{\label{fig:generic module diagram}}{
		\begin{tikzcd}[
			column sep={2.75em,between origins},
			ampersand replacement=\&,
			execute at end picture={
				\draw (O.south west) coordinate (B) edge [->, shorten <=-5pt] (O.south west |- ne.north east) edge [->, shorten <=-5pt] (O.south west -| ne.north east);
			}
			]
			\vdots                       \& \vdots           \& \vdots           \& [-3pt] |[alias=ne]| \phantom{0} \\
			M_{0,2} \rar\uar             \& M_{1,2} \rar\uar \& M_{2,2} \rar\uar \& \cdots                          \\
			M_{0,1} \rar\uar             \& M_{1,1} \rar\uar \& M_{2,1} \rar\uar \& \cdots                          \\
			|[alias=O]| M_{0,0} \rar\uar \& M_{1,0} \rar\uar \& M_{2,0} \rar\uar \& \cdots
		\end{tikzcd}
	}\hfill
	\subcaptionbox{\label{fig:the indecompose module}}{%
		\begin{tikzcd}[
			ampersand replacement=\&,
			execute at end picture={
				\draw (O.south west) coordinate (B) edge [->, shorten <=-5pt] (O.south west |- ne.north east) edge [->, shorten <=-5pt] (O.south west -| ne.north east);
			},
			baseline=(O.south west)
			]
			0                                             \& 0                                                                       \& 0                              \& [-5pt] |[alias=ne,alias=66]| \phantom{0} \\
			|[alias=04]| \F \rar[equal]\uar               \& |[alias=24]| \F \rar[equal]\uar                                         \& |[alias=44]|\F \rar\uar        \& 0                                        \\
			|[alias=02]| \F \rar["\Mtx{0\\1}"]\uar[equal] \& |[alias=22]| \F^2 \rar\uar\ar[ur,phantom,"\scriptstyle {(1,1)}" pos=.3] \& |[alias=42]|\F \rar\uar[equal] \& 0                                        \\
			|[alias=O]| 0 \rar\uar                        \& |[alias=20]| \F \rar[equal]\uar["\Mtx{1\\0}"']                          \& |[alias=40]|\F \rar\uar[equal] \& 0
		\end{tikzcd}
	}\hfill
	\subcaptionbox{\label{fig:module diagram of the indecomposable module}}{
		\begin{tikzpicture}[module diagram, y=1em, x=1em, baseline=0pt]
			\Axes[->, shorten <=-5pt](0,0)(6.5,6.5)
			\Grid[1](3){6}[1](3){6}
			\begin{scope}[blue]
				\filldraw[fill opacity=0.2] (2,0) -| (6,6) -| (0,2) -| cycle;
				\filldraw[fill opacity=0.2] (2,2) rectangle (4,4);
				\draw[->] (3,1.5) to["$\Mtx{1 \\ 0}$" {right, at start, anchor=north west, inner sep=1pt}] (3,2.5);
				\draw[->] (1.5,3) to["$\Mtx{0 \\ 1}$" {above, at start, anchor=-55, inner sep=1pt}] (2.5,3);
				\draw[->] (3.5,3.5) to["$\scriptstyle (1\:1)$" {left, at end, anchor=south east, inner sep=1pt}] (4.5,4.5);
			\end{scope}
		\end{tikzpicture}
	}\hfill
	\begin{tabular}[b]{@{}c@{}}
		\subcaptionbox{\label{fig:one-parameter module:1}}{
			\begin{tikzpicture}[module diagram, node font=\scriptsize, nodes={inner sep=0.25pt}]
				\path
				(-2,0) node (0) {$\cdots$}
				(0,0) node (1) {$\F$} edge[<-] (0)
				(2,0) node (2) {$\F^2$} edge[<-] (1)
				(4,0) node (3) {$\F^3$} edge[<-] (2)
				(6,0) node (4) {$\F^4$} edge[<-] (3)
				(8,0) node (5) {$\cdots$} edge[<-] (4);
				\path[use as bounding box] (-2,0) -- (8,10pt);
			\end{tikzpicture}
		} \\[2em]
		\subcaptionbox{\label{fig:one-parameter module:2}}{
		\begin{tikzpicture}[module diagram, node font=\scriptsize, nodes={inner sep=0.25pt}]
				\draw[->] (-2,0) -- (8,0);
				\draw[blue, yshift=10pt, thick] (0,0) node[generator]{} -- (8,0);
				\draw[blue, yshift= 8pt, thick] (2,0) node[generator]{} -- (8,0);
				\draw[blue, yshift= 6pt, thick] (4,0) node[generator]{} -- (8,0);
				\draw[blue, yshift= 4pt, thick] (6,0) node[generator]{} -- (8,0);
				\path[use as bounding box] (-2,0) -- (8,10pt);
		\end{tikzpicture}
	}
\end{tabular}
	\caption{Graphical illustration of one- and two-parameter modules.}
	\label{fig:module diagram}
\end{figure}
For $n \leq 2$, we visualize modules as in \zcref{fig:generic module diagram},
i.e., as a commutative grid of the module's components and structure morphisms.
Simplifying the notation further, we represent the dimension of the components by intensity of color, where areas left white correspond to dimension zero.
If necessary, we indicate by matrices the structure morphisms between these regions;
where no matrix is present, the structure morphisms are given by canonical inclusions or projections.
For example, the module in \zcref{fig:the indecompose module} is depicted as in \zcref{fig:module diagram of the indecomposable module}.
For one-parameter modules, we use analogous illustrations; see \zcref{fig:one-parameter module:1,fig:one-parameter module:2}.

\paragraph*{Free, flat and injective modules}
Let $\lZ \coloneqq \Z \cup \{-\infty\}$ and $\uZ \coloneqq \Z \cup \{\infty\}$.
For $z \in \lZ^n$ and $z' \in \uZ^n$, let $F(z)$ and $I(z')$ be the modules with
\begin{equation}
	F(z)_{v} = \begin{cases}
		\F & \text{if $z \leq v$}, \\
		0   & \text{otherwise,}
	\end{cases}
	\qquad \text{resp.} \qquad
	I(z')_{v} = \begin{cases}
		\F & \text{if $v \leq z'$}, \\
		0   & \text{otherwise}
	\end{cases}
\end{equation}
and the obvious structure morphisms.
A module $M \in \Pers{n}$ is \emph{flat} (resp.\ \emph{injective}) if there is an isomorphism
\begin{equation}
	\label{eq:def:basis}
	\smashoperator{\bigoplus_{z \in \rk M}} F(z) \xto{\beta} M \qquad\text{resp.}\qquad \smashoperator{\bigoplus_{z \in \rk M}} I(z) \xto{\beta} M
\end{equation}
for some multiset $\rk M \subseteq \lZ^n$ (resp.\ $\rk M \subseteq \uZ^n$);
see \zcref{fig:free-flat-inj-modules} for an illustration.
The multiset $\rk M$ is uniquely determined by $M$ and called the \emph{generalized graded rank} of $M$.
An isomorphism $\beta$ as in \eqref{eq:def:basis} is a \emph{generalized basis} of $M$.
An \emph{ordered generalized basis} of $M$ is a generalized basis of $M$, together with a total order on $\rk M$.
A flat module $M$ is \emph{free} and an injective module $M$ is \emph{cofree} if $\rk M \subseteq \Z^n$.
The generalized rank and a generalized basis of a free module are its graded rank and a basis in the usual sense.
A module $M$ is \emph{finitely generated} if there exists a surjective morphism $F \onto M$ for $F$ free of finite rank.

\begin{figure}
	\newcommand\Setup{
		\Axes(-1.5,-1.5)(5.5,5.5)
		\Grid[-1](0){5}[-1](0){5}
		\coordinate (g) at (3,3);
	}
	\begin{subfigure}{.2\linewidth}
		\centering
		\begin{tikzpicture}[module diagram]
			\Setup
			\FreeModule[blue](2,2){}
			\FreeModule[blue](1,4){}
			\FreeModule[blue](4,1){}
		\end{tikzpicture}
		\caption{free}
	\end{subfigure}%
	\begin{subfigure}{.2\linewidth}
		\centering
		\begin{tikzpicture}[module diagram]
			\Setup
			\FreeModule<x>[blue](g){}
			\FreeModule<y>[blue](g){}
			\FreeModule[blue](1,1){}
		\end{tikzpicture}
		\caption{flat} 
	\end{subfigure}%
	\begin{subfigure}{.2\linewidth}
		\centering
		\begin{tikzpicture}[module diagram]
			\Setup
			\FreeModule<xy>[blue](g){}
		\end{tikzpicture}
		\caption{flat-injective}
	\end{subfigure}%
	\begin{subfigure}{.2\linewidth}
		\centering
		\begin{tikzpicture}[module diagram]
			\Setup
			\coordinate (g) at (1,1);
			\InjModule<x>[blue](g){}
			\InjModule<y>[blue](g){}
			\InjModule[blue](3,3){}
		\end{tikzpicture}
		\caption{injective}
	\end{subfigure}%
	\begin{subfigure}{.2\linewidth}
		\centering
		\begin{tikzpicture}[module diagram]
			\Setup
			\InjModule[blue](4,1){}
			\InjModule[blue](3,3){}
			\InjModule[blue](1,4){}
		\end{tikzpicture}
		\caption{cofree}
	\end{subfigure}
	\caption{Free, flat (=projective), injective and cofree modules.
		Each quadrant, half plane or entire plane corresponds to one free, flat, injective or cofree indecomposable summand.}
	\label{fig:free-flat-inj-modules}
\end{figure}

For $M$ a module and $z \in \Z^n$, let $M\Shift{z}$ be the module with $M\Shift{z}_w \coloneqq M_{w+z}$.
We let $\IHom(M, N)$ be the module with
\[
	\IHom(M,N)_z \coloneqq \Hom(M, N\Shift{z}),
\]
and $M \otimes N$ be the module with
\begin{equation}
	\label{eq:def:tensor}
	(M \otimes N)_z = \smashoperator{\bigoplus_{v+w=z}} M_v \otimes_\F N_w/{\sim_\otimes},
\end{equation}
for the equivalence relation $\sim_\otimes$ generated by
\begin{equation}
	\label{eq:eq-rel-tensor}
	M_{v,v-u}(m) \otimes n \sim_\otimes m \otimes N_{w,w-u}(n)
\end{equation}
for all $m \in M_{v-u}$ and $n \in N_{w-u}$.
These definitions give rise to a right exact bifunctor $- \otimes -$ and a left exact bifunctor $\IHom(-,-)$ that fit into the usual adjunction.
A module is flat (resp.\ injective) if and only if it is flat (injective) in the categorial sense,
meaning that $- \otimes M$ (resp.\ $\IHom(-, M)$) is an exact functor \cite[Lemma~11.23, Theorem~11.30]{MillerSturmfels:2005}.

\paragraph*{Duals}
\phantomsection\label{sec:duals}
The \emph{Matlis dual} $M^*$ of a module $M$ is the module with $(M^*)_z = \Hom_\F(M_{-z}, \F)$.
The \emph{global dual} $M^\dagger$ of $M$ is the module with $(M^\dagger)_z = \Hom(M, F(-z))$.
The functor $(-)^*$ is an exact autoequivalence on $\Pers{n}$.
A module $M$ is flat (resp.\ free) if and only if $M^*$ is injective (resp.\ cofree) \cite[Lemma~11.23]{MillerSturmfels:2005}.

If $\beta$ is a generalized basis of a finite rank flat or injective module $M$ as in \eqref{eq:def:basis},
then $(\beta^*)^{-1}$ is a generalized basis of $M^*$.
If $M$ is flat, then $(\beta^\dagger)^{-1}$ is a generalized basis of $M^\dagger$.
We call both the \emph{dual basis} of $\beta$.

\paragraph*{Graded matrices}
\label{sec:graded-matrices}
For computational considerations, it is customary to describe morphisms of free and injective modules using graded matrices, which we outline in the following.
A \emph{graded matrix} $U$ consists of an underlying matrix $\u{U} \in \F^{l \times m}$,
together with \emph{row} and \emph{column grades} $\rg^U_i$, $\cg^U_j \in (\Z \cup \{\pm\infty\})^n$ for each $i \leq l$ and $j \leq m$.
We call $U$ \emph{valid} if $U_{ij} \neq 0$ only if $\rg^U_i \leq \cg^U_j$.
We call $U$ \emph{anti-valid} if $U_{ij} \neq 0$ only if $\rg^U_i \geq \cg^U_j$.
Graded matrices are called \emph{monomial matrices} in \cite{Miller:2020a}.

\begin{figure}
	\newcommand\Setup{
		\Axes(-2.5,-2.5)(2.5,2.5)
		\Grid[-2](-1){2}[-2](-1){2}
	}
	\subcaptionbox{$F\binom{1}{1} \into F\binom{-1}{-1}$\label{fig:flat-morphism}}{
		\begin{tikzpicture}[module diagram]
			\Setup
			\FreeModule[red](1,1){}
		\end{tikzpicture}
		$\into$
		\begin{tikzpicture}[module diagram]
			\Setup
			\FreeModule[red](1,1){}
			\FreeModule[blue](-1,-1){}
		\end{tikzpicture}
	}\hfill
	\subcaptionbox{$I\binom{1}{1} \into I\binom{-1}{-1}$\label{fig:inj-morphism}}{
		\begin{tikzpicture}[module diagram]
			\Setup
			\InjModule[blue](1,1){}
			\InjModule[red](-1,-1){}
		\end{tikzpicture}
		$\onto$
		\begin{tikzpicture}[module diagram]
			\Setup
			\InjModule[red](-1,-1){}
		\end{tikzpicture}
	}\hfill
	\subcaptionbox{$F\binom{-1}{-1} \to I\binom11$\label{fig:flat-injective-morphism}}{
		\begin{tikzpicture}[module diagram]
			\Setup
			\FreeModule[blue](0,0){}
			\filldraw[red,fill opacity=.2](-.5,-.5) rectangle (.5,.5);
		\end{tikzpicture}
		$\to$
		\begin{tikzpicture}[module diagram]
			\Setup
			\InjModule[blue](0,0){}
			\filldraw[red,fill opacity=.2](-.5,-.5) rectangle (.5,.5);
		\end{tikzpicture}
	}
	\caption{Morphisms between flat modules, injective modules, and from flat to injective modules.
		In all cases, the red area in the domain and codomain denotes the image of the morphism.}
	\label{fig:morphisms}
\end{figure}
We have $\Hom(F(z), F(z')) \cong \Hom(I(z), I(z')) \cong \begin{smallcases} \F & \text{if $z' \leq z$,} \\ 0 & \text{otherwise}\end{smallcases}$,
where a scalar $\lambda \in \F$ is identified with the injective morphism $F(z) \to F(z')$
and the surjective morphism $I(z) \to I(z')$ given by multiplication with $\lambda$ in nonzero components;
see also \zcref{fig:flat-morphism,fig:inj-morphism}.
Therefore, if $L$ and $M$ are both flat (or both injective) modules of finite generalized rank,
then after choosing ordered generalized bases,
we may identify $\Hom(L, M)$ with the vector space of valid graded matrices $U$ with $\rg^U = \rk L$ and $\cg^U = \rk M$.

Similarly, we have $\Hom(F(z'), I(z)) \cong \begin{smallcases} \F & \text{if $z' \leq z$} \\ 0 & \text{otherwise}\end{smallcases}$,
see also \zcref{fig:flat-injective-morphism}.
Thus, if $F$ is flat and $I$ injective, both of finite generalized rank,
then after choosing ordered generalized bases of both, we may identify $\Hom(F, I)$
with the vector space of anti-valid graded matrices $U$ with row grades $\rg^U = \rk I$ and $\cg^U = \rk F$.

\begin{example}
	Let $z_1 = (0,2)$, $z_2 = (1,1)$, $z_3 = (0,2)$ and $w_1 = (1,2)$, $w_2 = (2,1)$. Then
	\[
		\textstyle \Hom\bigl(\bigoplus_i F(w_i), \bigoplus_i F(z_i)\bigr)
		\cong \Hom\bigl(\bigoplus_i I(w_i), \bigoplus_i I(z_i)\bigr)
		\cong \biggl\{
			\begin{pNiceMatrix}[first-col,first-row,small]
				      & (1,2) & (2,1) \\
				(0,2) & *     & 0     \\
				(1,1) & *     & *     \\
				(2,1) & 0     & *
			\end{pNiceMatrix}
		\biggr\}
	\] and \[
		\textstyle \Hom\bigl(\bigoplus_i F(z_i), \bigoplus_i I(w_i)\bigr)
		\cong \biggl\{
			\begin{pNiceMatrix}[first-col,first-row,small]
				      & (0,2) & (1,1) & (2,0) \\
				(1,2) & *     & *     & 0     \\
				(2,1) & 0     & *     & *
			\end{pNiceMatrix}
		\biggr\},
	\]
	where $*$ stands for an arbitrary element from $\F$.
\end{example}

The \emph{graded transpose} $U^T$ of a graded matrix $U$ is the graded matrix $U^T$ with entries $\u{U^T} = \u{U}^T$,
row grades $\rg^{U^T}_i = -\cg^U_i$ and column grades $\cg^{U^T}_j = -\rg^U_j$.
For $z \in \Z^n$, the \emph{shift by $z$} of a graded matrix $U$ is the graded matrix $U\Shift{z}$ with $\u{U\Shift{z}} = \u{U}$,
row grades $\rg^{U\Shift{z}}_i = \rg^U_i - z$ and column grades $\cg^{U\Shift{z}}_j = \rg^U_j - z$,
where entries $\pm\infty$ stay invariant.
A graded matrix $U$ is valid if and only if $U^T$ and/or $U\Shift{z}$ is valid.
If the graded matrix $U$ represent a morphism $f\colon L \to M$ of (flat and/or injective) modules,
then $U\Shift{z}$ represents the morphism $f\Shift{z}\colon L\Shift{z} \to M\Shift{n}$.

\begin{lemma}[{\cite{companion-dualities}}]
	\label{thm:dualities-matrices}
	Let a graded matrix $U$ represent a morphism $f\colon L \to M$ of flat modules
	w.r.t.\ some ordered generalized bases $\beta$ and $\gamma$ of $L$ and $M$.
	Then $U^T$ represents the morphism $f^*\colon M^* \to L^*$ of injective modules
	and the morphism $f^\dagger\colon M^\dagger \to L^\dagger$ of flat modules
	w.r.t. the bases dual to $\gamma$ and $\beta$.
\end{lemma}

\paragraph*{Resolutions and presentations}
A \emph{flat \emph{(resp.\ \emph{free})} resolution} of a module $M \in \Pers{n}$
is a (potentially infinite) exact sequence $\dotsb \to F_1 \to F_0$ of flat (resp.\ free) modules,
such that there is an exact sequence $\dotsb \to F_1 \to F_0 \stackrel\varepsilon\to M \to 0$.
An \emph{injective \emph{(resp.\ \emph{cofree})} resolution} of $M$
is an exact sequence $I^0 \to I^1 \to \dotsb$ of injective (resp.\ cofree) modules,
such that there is an exact sequence $0 \to M \stackrel\eta\to I^0 \to I^1 \to \dotsb$.
The morphisms $\varepsilon$ and $\eta$ are called the \emph{augmentation morphisms} of the respective resolutions.
The morphism $F_1 \to F_0$ is also called a \emph{flat \emph{(resp.\ \emph{free})} presentation} of $M$,
and the morphism $I^0 \to I^1$ is also called an \emph{injective \emph{(resp.\ \emph{cofree})} presentation} of $M$.
Flat, free, injective or cofree presentations (or resolutions) can be represented by a valid graded matrix (resp.\ a sequence thereof).

The following gives a sufficient criterion for finite flat (even free) resolutions to exist.
The \emph{length} of a resolution $F_\bullet$ (or $I^\bullet$) is the maximal $\ell$ such that $F_\ell \neq 0$ (resp.\ $I^\ell \neq 0$).

\begin{theorem}[{Hilbert's Syzygy theorem \cite[Theorem~15.2]{Peeva:2011}, \cite[Corollary~19.7]{Eisenbud:1995}}]
	\label{thm:hilbert-syzygy}
	Every finitely generated module in $\Pers{n}$ has a free resolution of length at most $n$.
\end{theorem}

\begin{remark}
	\label{rmk:fd-length-n-resolution}
	Every resolution of a finite dimensional module in $\Pers{n}$ has as length at least $n$.
	This will follow from \zcref{thm:nu-resolutions} below.
\end{remark}

\begin{example}[Flat and injective resolutions]
	\label{ex:resolutions}
	Consider the indecomposable module
	\begin{equation}
		\label{eq:indec-module}
		M\ =\ \begin{tikzcd}[row sep=.9em, column sep=.9em, cramped, cells={font=\scriptsize}, baseline={(b.base)}]
			0 \rar     & 0 \rar                           & 0   \rar                                             & 0 \rar             & 0      \\
			0 \uar\rar & \F \uar\rar[equal]               & \F   \uar\rar[equal]                                 & \F \uar\rar        & 0 \uar \\
			0 \uar\rar & \F \uar[equal]\rar["\Mtx{0\\1}"] & \F^2 \uar\rar\ar[ur, phantom, "\scriptstyle{(1,1)}"] & \F \uar[equal]\rar &|[alias=b]| 0, \uar \\
			0 \uar\rar & 0 \uar\rar                       & \F   \uar["\Mtx{1\\0}"']\rar[equal]                  & \F \uar[equal]\rar & 0 \uar \\
			0 \uar\rar & 0 \uar\rar                       & 0 \uar\rar                                           & 0 \uar\rar         & 0 \uar
		\end{tikzcd}
	\end{equation}
	see also \zcref{fig:indec-module}.
	The sequence
	\def\I{\mathmakebox[\widthof{$F$}][r]{I}}
	\begin{equation}
		\label{eq:free-res}
		0 \to F\tbinom22 \oplus F\tbinom33
		\xto{\Mtx*[r]{0 & -1 \\ 1 & -1 \\ -1 & 0 \\ 0 & 1}} F\tbinom03 \oplus F\tbinom12 \oplus F\tbinom21 \oplus F\tbinom30
		\xto{\Mtx*[r]{1 & -1 & -1 & 0 \\ 0 & 1 & 1 & 1}} F\tbinom10 \oplus F\tbinom01
		\phantom{{}\to0}
	\end{equation}
	is a flat (even free) resolution of $M$; see \zcref{fig:free-res}. The sequence
	\begin{equation}
		\label{eq:inj-res}
		\phantom{0\to{}} \I\tbinom22 \oplus \I\tbinom33
		\xto{\Mtx*[r]{0 & -1 \\ 1 & -1 \\ -1 & 0 \\ 0 & 1}} \I\tbinom03 \oplus \I\tbinom12 \oplus \I\tbinom21 \oplus \I\tbinom30
		\xto{\Mtx*[r]{1 & -1 & -1 & 0 \\ 0 & 1 & 1 & 1}} \I\tbinom10 \oplus \I\tbinom01
		\to 0
	\end{equation}
	is an injective (even cofree) resolution of $M$; see \zcref{fig:inj-res}.
\end{example}

The following definition introduces a different notion of presentations, that mixes flat and injective modules:

\begin{figure}
	\tikzset{every module diagram/.append style={x=2.8mm,y=2.8mm}}%
	\newcommand\Setup{
		\Axes[->](-1.5,-1.5)(3.5,3.5)
		\Grid[-1](0){3}[-1](0){3}
	}%
	\centering
	\subcaptionbox{free resolution of $M$\label{fig:free-res}}{
		\begin{tikzpicture}[module diagram]
			\Setup
			\FreeModule[blue](3,3){};
			\FreeModule[blue](2,2){};
		\end{tikzpicture}
		$\into$
		\begin{tikzpicture}[module diagram]
			\Setup
			\FreeModule[blue](3,0){};
			\FreeModule[blue](2,1){};
			\FreeModule[blue](1,2){};
			\FreeModule[blue](0,3){};
		\end{tikzpicture}
		$\to$
		\begin{tikzpicture}[module diagram, remember picture]
			\Setup
			\FreeModule[blue](1,0){};
			\FreeModule[blue](0,1){};
			\coordinate (s) at (current bounding box.north east);
		\end{tikzpicture}
		$\stackrel\varepsilon\onto$%
	}%
	\subcaptionbox{$M$\label{fig:indec-module}}{%
		\begin{tikzpicture}[module diagram]
			\Setup
			\begin{scope}[blue]
				\filldraw[fill opacity=0.2] (.5,-.5) -| (2.5,2.5) -| (-.5,.5) -| cycle;
				\filldraw[fill opacity=0.2] (.5,.5) rectangle (1.5,1.5);
				\draw[->] (1,.25) to["$\binom10$" {right}] (1,.75);
				\draw[->] (.25,1) to["$\binom01$" {left}] (.75,1);
				\draw[->] (1.25,1.25) to["$\scriptstyle (1\:1)$" {above}] (1.75,1.75);
			\end{scope}
		\end{tikzpicture}%
	}%
	\subcaptionbox{cofree resolution \eqref{eq:inj-res} of $M$\label{fig:inj-res}}{%
		$\stackrel\eta\into$
		\begin{tikzpicture}[module diagram, remember picture]
			\Setup
			\begin{scope}[shift={(-1,-1)}]
				\InjModule[blue](3,3){};
				\InjModule[blue](2,2){};
			\end{scope}
			\coordinate (t) at (current bounding box.north west);
		\end{tikzpicture}
		\tikz[remember picture,overlay] \path (s) edge[->, bend left=5mm, shorten >=3pt, shorten <=3pt, font=\scriptsize, "$\phi$"] (t);
		$\to$
		\begin{tikzpicture}[module diagram]
			\Setup
			\begin{scope}[shift={(-1,-1)}]
				\InjModule[blue](3,0){};
				\InjModule[blue](2,1){};
				\InjModule[blue](1,2){};
				\InjModule[blue](0,3){};
			\end{scope}
		\end{tikzpicture}
		$\onto$
		\begin{tikzpicture}[module diagram]
			\Setup
			\begin{scope}[shift={(-1,-1)}]
				\InjModule[blue](1,0){};
				\InjModule[blue](0,1){};
			\end{scope}
		\end{tikzpicture}
	}
	\caption{The module $M$ from \eqref{eq:indec-module} in \zcref{ex:resolutions}, together with
		a \subref{fig:free-res} free, \subref{fig:inj-res} cofree resolution,
		and \subref{fig:indec-module} flat-injective presentation $\phi$ of $M$.}
\end{figure}

\begin{definition}[Flat-injective presentations {\cite[Definition~5.12]{Miller:2020a}}]
	A \emph{flat-injective presentation} of a module $M$ is a morphism $\phi\colon F \to I$ such that $F$ is flat, $I$ is injective and $M \cong \im \phi$.
\end{definition}

\begin{remark}\label{rmk:rank-invariant}
	For a $\Z^n$-graded matrix $U$, grades $z, z' \in \Z^n$ and binary relations $\diamond, \diamond' \in \{\leq, \geq\}$,
	define the ungraded matrix $U_{\diamond z, \diamond' z'} \coloneqq (U_{ij})_{\rg^U_i \diamond z, \cg^U_j \diamond' z'}$.
	Let $M \in \Pers{n}$ have totally finite dimension,
	and $U$, $V$ and $M$ be graded matrices representing a free, flat-injective and injective presentation of $M$, respectively, w.r.t.\ arbitrarily chosen bases.
	Then the structure morphisms of $M$ satisfy
	\[
	\im M_{z'z} \cong \coker U_{\geq z, \geq z'} \cong \im V_{\leq z', \geq z} \cong \ker W_{\leq z, \leq z'}.
	\]
	This also determines the components of $M$ since $M_z = \im M_{zz}$.
\end{remark}

\begin{example}[Continuation of \zcref{ex:resolutions}]
	\label{ex:resolutions:ctd}
	Composing the surjective augmentation morphism $\varepsilon\colon F_0 \onto M$ with the injective augmentation morphism $\eta\colon M \into I^0$
	gives a morphism $\phi \coloneqq \eta\epsilon\colon F_0 \to I^0$ whose image is isomorphic to $M$;
	see \zcref{fig:indec-module}.
	With respect to the generalized bases underlying \zcref{eq:free-res,eq:inj-res}, it is represented by the anti-valid graded matrix
	\begin{equation}
		\label{eq:flat-injective-res}
		\Phi = \begin{pNiceMatrix}[first-col,first-row]
			      & (1,0) & (0,1)        \\
			(1,1) & -1    & \phantom{-}0 \\
			(2,2) & -1    & -1
		\end{pNiceMatrix}.
	\end{equation}
	To see that this is a flat-injective presentation of $M$, consider the module $\im \phi$, whose components are the vector spaces
	\(
		(\im \phi)_z = \im\bigl((\Phi_{ij})_{\cg^U_j \leq z \leq \rg^U_i}\bigr)
	\),
	so $\im \phi$ is the module
	\[
		\tikzcdset{
			diagrams={row sep={2.6em,between origins},column sep={1.5em},cramped,}
		}
		\begin{tikzcd}
			       & [-2ex] 0                   & 0                                     & 0                        & [-2ex]  \\[-1.75ex]
			0 \rar & \im(-1) \rar\uar           & \im({-1}\:{-1}) \rar\uar              & \im({-1}\:{-1}) \rar\uar & 0       \\
			0 \rar & \im\Mtx*[r]{0\\-1}\rar\uar & \im\Mtx*[r]{-1 & 0\\-1 & -1} \rar\uar & \im({-1}\:{-1})\uar\rar  & 0       \\
			0 \rar & \im() \rar \uar            & \im\Mtx*[r]{-1\\-1} \rar\uar          & \im(-1) \rar\uar         & 0       \\[-1.75ex]
			       & 0 \uar                     & 0 \uar                                & 0 \uar                   &
		\end{tikzcd}
		\qquad
		\cong
		\qquad
		\begin{tikzcd}
			       & [-2ex] 0                        & 0                                                       & 0                 & [-2ex] \\[-1.75ex]
			0 \rar & \F\uar\rar[equal]               & \F \uar\rar[equal]                                      & \F\uar\rar        & 0      \\
			0 \rar & \F\uar[equal]\rar["\Mtx{0\\1}"] & \F^2 \ar[ur, phantom, "\scriptstyle (1\:1)"'] \uar \rar & \F\uar[equal]\rar & 0      \\
			0 \rar & 0\uar\rar                       & \F\uar["\Mtx{1\\0}"]\rar[equal]                         & \F\uar[equal]\rar & 0      \\[-1.75ex]
			       & 0 \uar                          & 0 \uar                                                  & 0 \uar            &
		\end{tikzcd}
	\]
	which is $M$ indeed.
	Here, $()$ denotes an empty matrix, which has $\im () = 0$.
\end{example}

\paragraph*{Homotopy equivalence}
A morphism $f_\bullet\colon C_\bullet \to D_\bullet$ of chain complexes is called a \emph{quasi\-/isomorphism} if the induced morphism $H_d(f_\bullet)$ in homology is an isomorphisms for every $d$.
In particular, viewing a module $M$ as chain complex concentrated in degree zero,
the augmentation morphism of a free (injective) resolution $F_\bullet$ (resp.\ $I^\bullet$) of $M$
is a quasi-isomorphism $F_\bullet \to M$ (resp.\ $M \to I^\bullet$) of chain complexes.
A chain complex $C_\bullet$ is \emph{acyclic} $H_i(C_\bullet) = 0$ for all $i$.
Two chain complexes $C_\bullet$, $D_\bullet$ are \emph{quasi-isomorphic} if there is are two quasi-isomorphisms $C_\bullet \leftarrow Z_\bullet \to D_\bullet$ for some complex $Z_\bullet$.

Two morphisms $f_\bullet, g_\bullet\colon C_\bullet \to D_\bullet$ of chain complexes are \emph{homotopic}
if there is a collection $s_\bullet = (s_d)_{d \in \Z}$ of morphisms $s_d\colon C_d \to D_{d+1}$ of modules,
such that $f_d-g_d = \partial^{D}_{d+1} s_d + s_d \partial^{C}_d$ for all $d$.
The collection $s_\bullet$ is called a called a \emph{homotopy} from $f_\bullet$ to $g_\bullet$.

Two morphisms $f_\bullet\colon C_\bullet \to D_\bullet$ and $g_\bullet\colon D_\bullet \to C_\bullet$ form a pair of mutually inverse \emph{homotopy equivalences}
if $f_\bullet g_\bullet$ and $\id_{D_\bullet}$ are homotopic, and $g_\bullet f_\bullet$ and $\id_{C_\bullet}$ are homotopic.
A chain complex $C_\bullet$ is \emph{contractible} if $C_\bullet \to 0$ is a homotopy equivalence;
in this case, a chain homotopy from $\id_{C_\bullet}$ to $0\colon C_\bullet \to C_\bullet$ is a \emph{contraction}.
In other words, a contraction is a collection $s_\bullet$ of morphism such that $\id_{C_d} = \partial^C_{d+1} s_d + s_{d-1} \partial^C_d$.
Every homotopy equivalence is a quasi-isomorphism; in particular, every contractible chain complex is acyclic.

The following is standard; we include a proof for reference.

\begin{lemma}
	\label{rmk:obtaining-contraction}
	Every quasi-isomorphism between bounded below complexes of flat modules is a chain homotopy equivalence.
	In particular, every acyclic bounded below complex of flat modules is contractible.
\end{lemma}
\begin{proof}
	If $C_\bullet$ is a bounded below chain complex of free modules that is known to be acyclic,
	then a chain contraction $s_\bullet$ of $C_\bullet$ can be obtained as follows.
	Assuming w.l.o.g.\ that $C_d = 0$ for $d < 0$,
	the assumed acyclicity of $C$ implies that $\partial_1$ is surjective.
	Choosing preimages of a basis of $C_0$ defines a morphism $s_0\colon C_0 \to C_1$.
	For $d > 0$, the assumption implies that for any basis element $e_i \in C_d$, we have $\partial_{d+1}^{-1}(e_i + \im s_{d-1}) \neq 0$.
	Therefore, choosing preimages for all $e_i$ defines a morphism $s_d\colon C_d \to C_{d+1}$.
	One checks easily that $s_\bullet$ is a contraction of $C_\bullet$.
\end{proof}

\paragraph*{Minimality}
A chain complex $C_\bullet$ of flat modules is called \emph{trivial} if it is isomorphic to a complex of the form $\dotsb \to 0 \to F(z) \xto{\cong} F(z) \to 0 \to \dotsb$,
and analogously for cochain complexes of injective modules.
A (co)chain complex of flat (or injective) is called \emph{minimal} if it does not contain any trivial complex as a direct summand.
A flat or injective resolution is called \emph{minimal} if is minimal as (co)chain complex.
A flat presentation $F_1 \to F_0 \to M$ is \emph{minimal} if it extends to a minimal flat resolution.

Recall that a graded matrix $U$ is valid if $U_{ij} \neq 0$ only if $\rg^U_i \leq \cg^U_j$.
It is \emph{minimal} if $U_{ij} \neq 0$ only if $\rg^U_i < \cg^U_j$.
The following observation is immediate:

\begin{lemma}
	A (co)chain complex of finite rank flat (resp.\ injective) modules is minimal
	if and only if the valid graded matrices representing its (co)boundary morphisms (with respect to any basis) are minimal.
\end{lemma}

\begin{example}
	The chain complex $C_\bullet$ given by the solid arrows in the first line of the following diagram is not minimal;
	the offending matrix entries are printed in bold.
	It is homotopy equivalent to the minimal chain complex $D_\bullet$ in the second line,
	and the vertical arrows form a pair of mutually inverse chain homotopy equivalences $f_\bullet$ and $g_\bullet$:
	\[
		\begin{tikzcd}[ampersand replacement=\&, column sep=large, row sep=normal]
			C_\bullet\colon\dar[shift left]{f_\bullet} \&[-3ex] F\tbinom22 \rar[swap]{\Mtx{1 \\ \bm{1} \\ 1}} \dar[shift left] \& F\tbinom11 \oplus F\tbinom22 \oplus F\tbinom22 \rar[swap]{\Mtx*[r]{0  & -1 & 1 \\ -1 & 1 & 0 \\ \bm{1} & 0 & -1}}\lar[bend right=15px, dashed, swap, end anchor=north east, start anchor=north west]{\Mtx{0&1&0}} \dar[shift left]{\Mtx{0&-1&1}}    \&[4ex] F\tbinom10 \oplus F\tbinom01 \oplus F\tbinom11  \dar[shift left]{\Mtx{1&0&0\\0&1&0}} \lar[bend right=15px, dashed, swap, end anchor=north east, start anchor=north west]{\Mtx{0&0&1\\0&0&0\\0&0&0}}\\
			D_\bullet\colon\uar[shift left]{g_\bullet} \&       0 \rar                                        \uar[shift left] \& F\tbinom22 \rar[swap]{\Mtx*[r]{1\\-1}}                                                                                                                                                                            \uar[shift left]{\Mtx{1\\0\\1}}   \&      F\tbinom10 \oplus F\tbinom01.                   \uar[shift left]{\Mtx{1&0\\0&1\\0&0}}
		\end{tikzcd}
	\]
	That $f_\bullet$ and $g_\bullet$ are mutually inverse chain homotopy equivalences is exhibited by
	the chain homotopy $s_\bullet$ with $\partial^C s+s\partial^C = \id_C-gf$ (dashed arrows).
	For the other composition, we have $f_\bullet g_\bullet = \id_{D_\bullet}$.
\end{example}

\begin{proposition}[{\cite[\S3]{FugacciKerberEtAl:2023}}]
	For every bounded-below chain complex $C_\bullet$ of finite-rank flat modules,
	there is a unique (up to isomorphism) minimal chain complex $\tilde{C}_\bullet$ and trivial chain complex $T_\bullet$
	such that $C_\bullet \cong \tilde{C}_\bullet \oplus T_\bullet$.
\end{proposition}

The projection and inclusion $C_\bullet \rightleftarrows \tilde{C}_\bullet$ are homotopy equivalences.

\begin{corollary}[{\cites[Theorem~20.2]{Eisenbud:1995}[Theorem~7.5]{Peeva:2011}}]
	Every finitely generated module $M \in \Pers{n}$ has a unique (up to isomorphism) minimal free resolution $F_\bullet$.
	Every free resolution of $M$ is isomorphic to $F_\bullet \oplus T_\bullet$ for a trivial complex $T_\bullet$.
\end{corollary}

The minimal free resolution of a finite dimensional module in $\Pers{n}$ has length precisely $n$;
cf.~\zcref{rmk:fd-length-n-resolution}.

\paragraph*{Dualities revisited}
Recall from \hyperref[sec:duals]{above} the Matlis dual $(-)^*$ and the global dual $(-)^\dagger$.
Recall that $(-)^*$ sends flat modules to injectives and vice versa, and $(-)^\dagger$ sends flat modules to flat modules.
Let $\one \coloneqq (1,\dotsc,1) \in \Z^n$.
The \emph{Nakayama functor} is the covariant functor $\nu$ with $\nu M \coloneqq (M^\dagger)^*\Shift{\one}$.
Let $\Proj{\Pers{n}}$ and $\Inj{\Pers{n}}$ be the full subcategories of $\Pers{n}$ consisting of finite rank flat and finite rank injective modules, respectively.
The Nakayama functor restricts to an equivalence of categories $\Proj{\Pers{n}} \to \Inj{\Pers{n}}$,
with quasi-inverse $\nu' M \coloneqq (M\Shift{-\one}^*)^\dagger$.
If $M$ is flat of finite generalized rank, then $\nu M$ is injective of generalized rank $\Set{-z-\one; z \in \rk M}$.

We call a chain complex $C_\bullet$ \emph{eventually acyclic} if $H_d(C_\bullet)$ is finite dimensional for all $d$.
If $C_\bullet$ is a chain complex, denote by $C_\bullet[i]$ the shifted chain complex with $(C_\bullet[i])_j = C_{i+j}$.
The following theorem is a special case of a graded version Greenlees--May duality:

\begin{theorem}[{\cite{companion-dualities}}]
	\label{thm:local-duality}
	If $C_\bullet$ is an eventually acyclic chain complex of finite rank free modules,
	then $C_\bullet$ and $\nu C_\bullet\HShift{n}$ are naturally quasi-isomorphic.
\end{theorem}

Let $F_\bullet\colon \dotsb F_2 \to F_1 \to F_0$ be a flat resolution of some module $M$.
Applying the contravariant functor $(-)^*$ to $F_\bullet$ yields a sequence $(F_\bullet)^*\colon F_0^* \to F_1^* \to F_2^* \to \dotsb$ of injective modules.
Exactness of $(-)^*$ implies that $(F_\bullet)^*$ is an injective resolution of $M^*$.
Observe that we may naturally view a chain complex $C_\bullet$ as a cochain complex $C^\bullet$ with $C^d = C_{-d}$.
Then \zcref{thm:local-duality} has the following corollary.

\begin{corollary}[{\cites{companion-dualities}[cf.][Example~14.5.18]{BrodmannSharp:2012}}]
	\label{thm:nu-resolutions}
	If $F_\bullet$ is a free resolution of a finite dimensional module $M \in \Pers{n}$,
	then $\nu F_\bullet\HShift{n}$ is an injective resolution of $M$.
\end{corollary}

Per \zcref{thm:dualities-matrices}, the injective resolution $\nu F_\bullet\HShift{n}$ of $M$
is represented (for suitable bases) by the same graded matrices (up to a grading shift by $\one$) as $F_\bullet$.

\begin{example}
	Recall the module $M$ from \zcref{ex:resolutions}.
	Its free and injective resolutions \zcref{eq:free-res,eq:inj-res} are mapped to each other by $\nu$ and $\nu'$.
	Thus, they can be represented by the same graded matrices,up to a degree shift by $\one$.
\end{example}

\begin{remark}
	\label{rmk:fd-necessary}
	Finite dimensionality is crucial for \zcref{thm:nu-resolutions}:
	The module
	\[
		N\ =\ \begin{tikzcd}[row sep=.9em, column sep=.9em, cramped, cells={font=\scriptsize}, baseline={(b.base)}]
			\vdots     & \vdots                           & \vdots                                               & \vdots                    &         \\
			0 \uar\rar & \F \uar[equal]\rar[equal]        & \F   \uar[equal]\rar[equal]                          & \F \uar[equal]\rar[equal] & \cdots  \\
			0 \uar\rar & \F \uar[equal]\rar["\Mtx{0\\1}"] & \F^2 \uar\rar\ar[ur, phantom, "\scriptstyle{(1,1)}"] & \F \uar[equal]\rar[equal] &|[alias=b]| \cdots  \\
			0 \uar\rar & 0 \uar\rar                       & \F   \uar["\Mtx{1\\0}"']\rar[equal]                  & \F \uar[equal]\rar[equal] & \cdots  \\
			0 \uar\rar & 0 \uar\rar                       & 0 \uar\rar                                           & 0 \uar\rar                & \cdots
		\end{tikzcd}
	\]
	depicted in \zcref{fig:indec-module-unbounded} has the free and injective resolution
	\begin{figure}
		\tikzset{every module diagram/.append style={x=3.5mm,y=3.5mm,baseline={(b)}}}%
		\newcommand\Setup{
			\Axes[->](-1.5,-1.5)(2.5,2.5)
			\Grid[-1](0){2}[-1](0){2}
			\coordinate (b) at (0,0.25);
			\path[use as bounding box] (ninf) -- (inf);
		}%
		\centering
		\subcaptionbox{\label{fig:free-res-unbounded}}{
			\begin{tikzpicture}[module diagram]
				\Setup
				\FreeModule[blue](2,2){};
			\end{tikzpicture}
			$\into$
			\begin{tikzpicture}[module diagram]
				\Setup
				\FreeModule[blue](2,1){};
				\FreeModule[blue](1,2){};
			\end{tikzpicture}
			$\to$
			\begin{tikzpicture}[module diagram]
				\Setup
				\FreeModule[blue](0,-1){};
				\FreeModule[blue](-1,0){};
			\end{tikzpicture}
			$\stackrel\varepsilon\onto$%
		}%
		\subcaptionbox{\label{fig:indec-module-unbounded}}{%
			\begin{tikzpicture}[module diagram]
				\Setup
				\begin{scope}[blue]
					\fill[fill opacity=0.2] (.5,-.5) coordinate (c1) -| (inf) -| (-.5,.5) coordinate (c2) -| cycle;
					\draw (c1 -| inf) -- (c1) |- (c2) -- (c2 |- inf);
					\filldraw[fill opacity=0.2] (.5,.5) rectangle (1.5,1.5);
					\draw[->] (1,.25) to["$\binom10$" {right}] (1,.75);
					\draw[->] (.25,1) to["$\binom01$" {left}] (.75,1);
					\draw[->] (1.25,1.25) to["$\scriptstyle (1\:1)$" {above}] (1.75,1.75);
				\end{scope}
			\end{tikzpicture}%
		}%
		\subcaptionbox{\label{fig:inj-res-unbounded}}{%
			$\stackrel\eta\into$
			\begin{tikzpicture}[module diagram]
				\Setup
				\InjModule<x,y>[blue](3,3){};
				\InjModule[blue](1,1){};
			\end{tikzpicture}
			$\to$
			\begin{tikzpicture}[module diagram]
				\Setup
				\InjModule<x>[blue](2,-1){};
				\InjModule[blue](1,0){};
				\InjModule[blue](0,1){};
				\InjModule<y>[blue](-1,2){};
			\end{tikzpicture}
			$\onto$
			\begin{tikzpicture}[module diagram]
				\Setup
				\InjModule[blue](0,-1){};
				\InjModule[blue](-1,0){};
			\end{tikzpicture}
		}
		\caption{An \subref{fig:free-res-unbounded} flat and \subref{fig:inj-res-unbounded} injective resolution
			of \subref{fig:indec-module-unbounded} the module $N$ from \eqref{rmk:fd-necessary}.}
		\label{fig:thm-fails-for-non-fd}
	\end{figure}
	depicted in \zcref{fig:free-res-unbounded,fig:inj-res-unbounded}.
	We see that both resolutions have different generalized rank.
\end{remark}

\section{Colimits and the Čech complex}
If $F_\bullet$ is an eventually acyclic complex of free modules, then 
according to \zcref{thm:local-duality}, the complexes $F_\bullet$ and $\nu F_\bullet[n]$ are quasi-isomorphic,
i.e., there exists a complex $\tilde{\Omega}_\bullet$ such that there are quasi-isomorphisms $F_\bullet \xleftarrow{\simeq} \tilde{\Omega}_\bullet \xto{\simeq} \nu C_\bullet[n]$.
In this section, we will introduce a suitable complex $\tilde{\Omega}_\bullet$, which is also used in the proof of \zcref{thm:local-duality}.
We will later see that for this $\tilde{\Omega}_\bullet$, the quasi-isomorphism $\tilde{\Omega}_\bullet \to F_\bullet$ is a homotopy equivalence.

\subsection{(Co)limit constructions}
For $n \in \N$, let $\IntSet{n} \coloneqq \{1,\dotsc,n\}$,
and for $k \in \N$, let $\binom{\IntSet{n}}{k} = \Set{Q \subseteq \IntSet{n}; \abs{Q} = k}$.

\begin{definition}
	\label{def:colim}
	For $Q \subseteq \IntSet{n}$,
	let \(p_Q\colon \Z^{n} \to \Z^{n-\abs{Q}}\) be the function that forgets the components indexed by $Q$.
	We define the functors
	\begin{alignat*}{4}
		\lim_Q   & \colon & \Pers{n}           & \to \Pers{(n-\abs{Q})},\qquad\qquad & (\lim_Q M)_z   & \coloneqq \lim_{w \in p_Q^{-1}(z)} M_w,   \\
		\Delta_Q & \colon & \Pers{(n-\abs{Q})} & \to \Pers{n},                       & (\Delta L)_w   & \coloneqq L_{p_Q(w)},                     \\
		\colim_Q & \colon & \Pers{n}           & \to \Pers{(n-\abs{Q})},             & (\colim_Q M)_z & \coloneqq \colim_{w \in p_Q^{-1}(z)} M_w,
	\end{alignat*}
	The structure morphisms of $\lim_Q M$, $\Delta_Q N$ and $\colim_Q M$ are defined analogously.
	We further define $\Lim_Q \coloneqq \Delta_Q \lim_Q$, $\Colim_Q \coloneqq \Delta_Q \colim_Q$.
\end{definition}

We will not need the functors $\lim_Q$ and $\Lim_Q$ in this paper; we include them nevertheless for completeness.
There are natural morphisms
\[
	\Lim_Q M \xto{\varepsilon^Q_M} M \xto{\eta^Q_M} \Colim_Q M
\]
for every $Q$ and $M$.
Further, for $Q \subseteq Q'$, there is a natural morphism $\rho^{Q',Q}_M\colon \Colim_Q M \to \Colim_{Q'} M$.

\begin{example}
	Let $n = 3$, let $M \in \Pers{3}$ be a module, and let $Q = \{2\}$.
	Then $\colim_Q M$ is the module in $\Pers{2}$ with $M_{(z_1,z_3)} = \colim_{w \in \Z} M_{(z_1,w,z_3)}$,
	and $\Colim_Q M$ is the module in $\Pers{2}$ with $M_{(z_1,z_2,z_3)} = \colim_{w \in \Z} M_{(z_1,w,z_3)}$,
	irrespective of the value of $z_2$.
\end{example}

\begin{lemma}
	\leavevmode
	\begin{enumerate}
		\item The pairs $(\lim_Q, \Delta_Q)$ and $(\Delta_Q, \colim_Q)$ both are adjoint pairs of functors.
		\item \label{thm:colim-exact} The functors $\Delta_Q$, $\colim_Q$ and $\Colim_Q$ are exact.
		\item The functors $\lim_Q$ and $\Lim_Q$ are exact when restricted to modules $M$ with $M_z$ constant for $z$ sufficiently small.
	\end{enumerate}
\end{lemma}
\begin{proof}
	The adjunctions $(\lim_Q, \Delta_Q)$ and $(\Delta_Q, \colim_Q)$ are induced by the respective adjunctions of (co)limits of vector spaces.
	Exactness of $\Delta_Q$ follows from $\Delta_Q$ being both a left and right adjoint.
	Exactness of $\colim_Q$ and $\Colim_Q$ follows because $\colim_Q$ is a colimit over the directed system $\Z^{\abs{Q}}$ of modules in $\Pers{(n-\abs{Q})}$.
	Exactness of $\lim_Q$ and $\Lim_Q$ follows from the Mittag-Leffler condition.
\end{proof}

\begin{figure}
	\subcaptionbox{$\lim_{\{1\}}$ and $\lim_{\{2\}}$\label{fig:lim-preserves-inj}}{
		\begin{tikzpicture}[module diagram]
			\Axes[shorten <=-3pt, shorten >=-3pt, ->](-2.5,-2.5)(2.5,2.5)
			\Grid[-2](-1){2}[-2](-1){2}
			\coordinate (g1) at (1,-1);
			\coordinate (g2) at (-1,1);
			\coordinate (g3) at (0,0);
			\InjModule[blue](g1){}
			\InjModule[blue](g2){}
			\InjModule<x>[blue](g3){}
			\begin{scope}[red]
				\scoped[transform canvas={shift={(0,-1.00)}}] \draw[black, ->, shorten <=-3pt, shorten >=-3pt] (ninf) -- (ninf -| inf);
				\scoped[transform canvas={shift={(0,-1.50)}}] \draw[thick] ([shift={(.5,0)}]g1 |- ninf) -- (ninf);
				\scoped[transform canvas={shift={(0,-1.75)}}] \draw[thick] ([shift={(.5,0)}]g2 |- ninf) -- (ninf);
				\scoped[transform canvas={shift={(0,-2.00)}}] \draw[thick] (ninf -| inf) -- (ninf);
				\scoped[transform canvas={shift={(-1.00,0)}}] \draw[black, ->, shorten <=-3pt, shorten >=-3pt] (ninf) -- (ninf |- inf);
				\scoped[transform canvas={shift={(-1.50,0)}}] \draw[thick] ([shift={(0,.5)}]g2 -| ninf) -- (ninf);
				\scoped[transform canvas={shift={(-1.75,0)}}] \draw[thick] ([shift={(0,.5)}]g1 -| ninf) -- (ninf);
				\scoped[transform canvas={shift={(-2.00,0)}}] \draw[thick] ([shift={(0,.5)}]g3 -| ninf) -- (ninf);
			\end{scope}
			\path[use as bounding box] (-5.5,-4.5) to (5.5,4.5);
		\end{tikzpicture}
	}\hfill
	\subcaptionbox{$\Delta_{\{1\}}$\label{fig:diag-preserves-inj}}{
		\begin{tikzpicture}[module diagram]
			\Axes[shorten <=-3pt, shorten >=-3pt, ->](-2.5,-2.5)(2.5,2.5)
			\Grid[-2](-1){2}[-2](-1){2}
			\coordinate (g1) at (1,-1);
			\coordinate (g2) at (-1,1);
			\coordinate (g3) at (0,0);
			\InjModule<x>[red](g1){}
			\InjModule<x>[red](g2){}
			\InjModule<x>[red](g3){}
			\begin{scope}[blue]
				\scoped[transform canvas={shift={(1.00,0)}}] \draw[black, ->, shorten <=-3pt, shorten >=-3pt] (ninf -| inf) -- (inf);
				\scoped[transform canvas={shift={(1.50,0)}}] \draw[thick] ([shift={(0,.5)}]g2 -| inf) -- (ninf -| inf);
				\scoped[transform canvas={shift={(1.75,0)}}] \draw[thick] ([shift={(0,.5)}]g1 -| inf) -- (ninf -| inf);
				\scoped[transform canvas={shift={(2.00,0)}}] \draw[thick] ([shift={(0,.5)}]g3 -| inf) -- (ninf -| inf);
			\end{scope}
			\path[use as bounding box] (-2.5,-4.5) to (4.5,4.5);
		\end{tikzpicture}
	}\hfill
	\subcaptionbox{$\Delta_{\{2\}}$\label{fig:diag-preserves-flat}}{
		\begin{tikzpicture}[module diagram]
			\Axes[shorten <=-3pt, shorten >=-3pt, ->](-2.5,-2.5)(2.5,2.5)
			\Grid[-2](-1){2}[-2](-1){2}
			\coordinate (g1) at (1,-1);
			\coordinate (g2) at (-1,1);
			\coordinate (g3) at (0,0);
			\FreeModule<y>[red](g1){}
			\FreeModule<y>[red](g2){}
			\FreeModule<xy>[red](g3){}
			\begin{scope}[blue]
				\scoped[transform canvas={shift={(0,1.00)}}] \draw[black, ->, shorten <=-3pt, shorten >=-3pt] (ninf |- inf) -- (inf);
				\scoped[transform canvas={shift={(0,1.50)}}] \draw[thick] ([shift={(-.5,0)}]g1 |- inf) -- (inf);
				\scoped[transform canvas={shift={(0,1.75)}}] \draw[thick] ([shift={(-.5,0)}]g2 |- inf) -- (inf);
				\scoped[transform canvas={shift={(0,2.00)}}] \draw[thick] (ninf |- inf) -- (inf);
			\end{scope}
			\path[use as bounding box] (-2.5,-4.5) to (2.5,4.5);
		\end{tikzpicture}
	}\hfill
	\subcaptionbox{$\colim_{\{1\}}$ and $\colim_{\{2\}}$\label{fig:colim-preserves-flat}}[\widthof{(a) $\colim_{\{1\}}$ and $\colim_{\{2\}}$}]{
		\begin{tikzpicture}[module diagram]
			\Axes[shorten <=-3pt, shorten >=-3pt, ->](-2.5,-2.5)(2.5,2.5)
			\Grid[-2](-1){2}[-2](-1){2}
			\coordinate (g1) at (1,-1);
			\coordinate (g2) at (-1,1);
			\coordinate (g3) at (0,0);
			\FreeModule[blue](g1){}
			\FreeModule[blue](g2){}
			\FreeModule<x>[blue](g3){}
			\begin{scope}[red]
				\scoped[transform canvas={shift={(0,1.00)}}] \draw[black, ->, shorten <=-3pt, shorten >=-3pt] (ninf |- inf) -- (inf);
				\scoped[transform canvas={shift={(0,1.50)}}] \draw[thick] ([shift={(-.5,0)}]g1 |- inf) -- (inf);
				\scoped[transform canvas={shift={(0,1.75)}}] \draw[thick] ([shift={(-.5,0)}]g2 |- inf) -- (inf);
				\scoped[transform canvas={shift={(0,2.00)}}] \draw[thick] (ninf |- inf) -- (inf);
				\scoped[transform canvas={shift={(1.00,0)}}] \draw[black, ->, shorten <=-3pt, shorten >=-3pt] (ninf -| inf) -- (inf);
				\scoped[transform canvas={shift={(1.50,0)}}] \draw[thick] ([shift={(0,-.5)}]g2 -| inf) -- (inf);
				\scoped[transform canvas={shift={(1.75,0)}}] \draw[thick] ([shift={(0,-.5)}]g1 -| inf) -- (inf);
				\scoped[transform canvas={shift={(2.00,0)}}] \draw[thick] ([shift={(0,-.5)}]g3 -| inf) -- (inf);
			\end{scope}
			\path[use as bounding box] (-5.5,-4.5) to (5.5,4.5);
		\end{tikzpicture}
	}
	\caption{
		Images (red) of some modules (blue) under $\colim_Q$, $\lim_Q$ and $\Delta_Q$.
		\subref{fig:lim-preserves-inj} $\lim_Q$ preserves injective and cofree modules,
		\subref{fig:colim-preserves-flat} $\colim_Q$ preserves flat and free modules, and
		\subref{fig:diag-preserves-flat}, {fig:diag-preserves-inj} $\Delta_Q$ preserves both flat and injective modules.}
	\label{fig:lim-colim-diag-preserves}
\end{figure}

\begin{remark}
	The functors defined above map free, flat, injective and cofree modules to:
	\begin{center}
		\begin{tabular}{c|ccccc}
			\toprule
			          & $\lim_Q$  & $\Lim_Q$  & $\Delta_Q$ & $\Colim_Q$  & $\colim_Q$  \\
			\midrule
			  free    &     0     &     0     &    flat    &    flat     &    free     \\
			  flat    &  (flat)   &  (flat)   &    flat    &    flat     &    flat     \\
			injective & injective & injective & injective  & (injective) & (injective) \\
			 cofree   &  cofree   & injective & injective  &      0      &      0\\
			\bottomrule
		\end{tabular}
	\end{center}
	Types in parenthesized can also be sent to zero; see \zcref{fig:lim-colim-diag-preserves}.
\end{remark}

The following justifies that we call a complex $C_\bullet$ \emph{eventually acyclic} if $H_q(C_\bullet)$ is finite dimensional for all $q$:

\begin{lemma}
	\label{thm:eventually-acyclic}
	If a chain complex $C_\bullet$ is eventually acyclic, then $\colim_Q C_\bullet = 0$ is acyclic for all $Q \neq \emptyset$.
\end{lemma}
\begin{proof}
	The functor $\Colim_Q$ is exact by \zcref{thm:colim-exact}, so $H_q(\colim_Q C_\bullet) \cong \colim_Q H_q(C_\bullet)$.
	If $H_q(C_\bullet)$ is finite dimensional, then for each $z$, the colimit $\bigl(\colim_Q H_q(C_\bullet)\bigr)_z$
	is a colimit over a diagram of vector spaces of which almost all are zero, hence the colimit is zero, too.
\end{proof}

\subsection{The Čech complex}
For $Q \subseteq \IntSet{n}$, let $e_Q \in \lZ^n$ be the vector with $(e_Q)_i = \begin{smallcases} -\infty & \text{if $i \in Q$,}\\ 0 & \text{otherwise,}\end{smallcases}$
and consider the module
\[
	\Omega_d = \smashoperator{\bigoplus_{Q \in \binom{\IntSet{n}}{n-d}}} F(e_Q).
\]
\begin{lemma}
	For every $M$, there is a natural isomorphism
	\[
		\Omega_d \otimes M \cong \smashoperator{\bigoplus_{Q \in \binom{\IntSet{n}}{n-d}}} \Colim_Q M.
	\]
\end{lemma}
\begin{proof}
	We show that there is a natural isomorphism $F(e_Q) \otimes M \cong \Colim_Q M$.
	The claim then follows from the fact that $\otimes$ commutes with finite direct sums.
	Without loss of generality, assume that $Q = \{k+1,\dotsc,n\}$ for some $0 \leq k < n$.
	A standard construction for the colimit of a diagram of vector space yields that for $w \in \Z^n$, we have
	\[
		(\Colim_Q M)_w = \bigl(\smashoperator{\bigoplus_{z_{k+1},\dotsc,z_k}} M_{(w_1,\dotsc,w_k,z_{k+1},\dotsc,z_k)}\bigr)\bigm/{\sim}
	\]
	where $x \sim M_{w,v}(x)$ for any $v \leq w$ and $x \in M_v$.
	For $w \in \Z^n$, let ${\downarrow}_Q(w) \coloneqq \Set{z \in \Z^n; \text{$z_i \leq w_i$ for all $i \in Q$}}$.
	One checks that
	\[
		(\Colim_Q M)_w = \bigl(\smashoperator{\bigoplus_{z \in {\downarrow}_Q(w)}} M_z\bigr)/{\sim}.
	\]
	Recall that
	\[
		(F(e_Q) \otimes M)_w = \Bigl(\bigoplus_{z} F(e_Q)_{w-z} \otimes_\F M_z\Bigr)/{\sim_\otimes},
	\]
	where $\sim_\otimes$ is the equivalence relation from \eqref{eq:def:tensor}.
	By definition, we have $F(e_Q)_{w-z} = \F$ if $w_i \geq z_i$ for all $i \in Q$, and zero otherwise.
	This implies that
	\[
		(F(e_Q) \otimes M)_w = \bigl(\smashoperator{\bigoplus_{z \in {\downarrow}_Q(w)}} \F \otimes_\F M_z\bigr)\bigm/{\sim_\otimes}.
	\]
	For $v \in \Z^N$, let $1^{(v)}$ denote the element $1 \in F(e_Q)_v$.
	Let $u \geq 0$ and $v, w \in \Z^n$.
	For $x \in M_w$, we have $1^{(u+v)} \otimes x \sim_\otimes 1^{(v)} \otimes M_{u+w,w}(x)$.
	Combining this with the isomorphism $M_w \mapsto \F \otimes_\F M_w$, $x \mapsto 1\otimes x$
	shows that $(F(e_Q) \otimes M)_w \cong (\Colim_Q M)_w$ naturally in $M$ and $w$.
	Using naturality in $w$, we obtain that $F(e_Q) \otimes M \cong \Colim_Q M$ as modules.
\end{proof}

For $Q \subseteq Q'$, let $\iota_{Q',Q}\colon F(e_Q) \into F(e_{Q'})$ be the canonical inclusion.

\begin{definition}
	\label{def:koszul-complex}
	The \emph{Čech complex} in $\Pers{n}$ is the chain complex $\Omega_\bullet$ of flat modules $\Omega_d$,
	with boundary morphisms $\kappa_d = \sum_{Q = \{q_1 < \dotsc < q_{n-d}\}}\sum_k (-1)^k\iota_{Q\setminus\{q_k\},Q}\colon \Omega_d \to \Omega_{d-1}$.
\end{definition}

For example, for $n = 2$, we obtain the chain complex depicted in \zcref{fig:koszul-complex}.
The Čech complex, which is commonly referred to by this name in algebraic geometry, is not to be confused with the Čech complex of a point cloud
considered in persistent homology.

\begin{figure}
	\centering
	\tikzset{
		every module diagram/.append style={baseline=(b)}
	}
	\newcommand{\Setup}{
		\Axes(-2.5,-2.5)(2.5,2.5)
		\Grid[-2](-1){2}[-2](-1){2}
		\coordinate (g) at (0,0);
		\coordinate (b) at (0,-1);
	}
	$
	0 \to
	\begin{tikzpicture}[module diagram]
		\Setup
		\FreeModule[blue](g){}
		\node[caption] {\Omega_2};
	\end{tikzpicture}
	\xinto{\Mtx{1\\1}}
	\begin{tikzpicture}[module diagram]
		\Setup
		\FreeModule<x>[blue](g){}
		\FreeModule<y>[blue](g){}
		\node[caption] {\Omega_1};
	\end{tikzpicture}
	\xto{(-1\ 1)}
	\begin{tikzpicture}[module diagram]
		\Setup
		\FreeModule<xy>[blue](g){}
		\node[caption] {\Omega_0};
	\end{tikzpicture}
	\onto
	\begin{tikzpicture}[module diagram]
		\Setup
		\InjModule[blue]($(g)-(1,1)$){}
		\node[caption] {I(-\one)};
	\end{tikzpicture}
	\to 0$
	\caption{The complex $\Omega_\bullet$ from \zcref{def:koszul-complex}.}
	\label{fig:koszul-complex}
\end{figure}

\begin{lemma}
	\label{thm:koszul}
	The complex $\Omega_\bullet$ is a flat resolution of $I(-\one)$,
	where the augmentation morphism
	\begin{equation}
		\label{eq:total-complex-augmentation-morphism}
		\varepsilon\colon \Omega_0 = F(e_{\IntSet{n}}) \onto I(-\one)
	\end{equation}
	has components $(\varepsilon)_z = \id_{\F}$ for all $z \leq -\one$.
\end{lemma}
\begin{proof}
	One checks that $H_i(\Omega_q) = 0$ for all $q > 0$, so $\Omega_\bullet$ is a flat resolution of some module $M \coloneqq H_0(\Omega_\bullet)$.
	By simple dimension counting, one sees $M_z = 0$ for all $z \nleq -\one$ and $M_z = \F$ for $z \leq -\one$.
	Again by dimension counting, we see $(\im\kappa_1)_z = 0$ for all $z \leq 0$, which implies that $M_{z'z} = \id_\F$ for all $z \leq z \leq -\one$.
\end{proof}

See \zcref{fig:koszul-complex} for an illustration.
Recall that the functor $\otimes$ commutes with finite direct sums.
Therefore, if $M$ is a finite rank flat module,
then $\Omega_\bullet M \coloneqq \Omega_\bullet \otimes M$ is a flat resolution of the injective module $\nu M$
with augmentation morphism $\varepsilon_M \coloneqq \varepsilon \otimes M\colon \Omega_0 M \to \nu M$.
There is a morphism $\varpi_M \coloneqq \varpi \otimes M\colon \Omega_n M \to M$.
Viewing the modules $M$ and $\nu M$ as chain complexes concentrated in degree zero,
we obtain a diagram of chain complexes
\begin{equation}
	\label{eq:F-Omega-nuF}
	\begin{tikzcd}[row sep=small]
		M[-n]\colon                                                               & M                                   &             &                         \\
		\Omega_\bullet  M \uar["\varpi_M"] \dar["\varepsilon_M"', "\simeq"]\colon & \Omega_n \otimes M \rar \uar[equal] & \dotsb \rar & \Omega_0 \otimes M \dar \\
		\nu M\colon                                                               &                                     &             & \nu M
	\end{tikzcd}
\end{equation}
where $\varepsilon_M$ is a quasi-isomorphism.

Let $(F_\bullet, \partial^F_\bullet)$ be a chain complex of free modules.
Then $\Omega_\bullet F_\bullet \coloneqq \Omega_\bullet \otimes F_\bullet$ is a double complex
with differentials
\begin{equation}
	\label{eq:kappa-partial}
	\kappa_{ij} \coloneqq \kappa_i \otimes \id_{F_j} \colon \Omega_i F_j \to \Omega_{i-1} F_j,\qquad
	\partial_{ij} \coloneqq \Omega_i \partial^{F}_j \colon \Omega_i F_j \to \Omega_i F_{j-1};
\end{equation}
i.e., the maps satisfy $\kappa_{i-1,j} \kappa_{i,j} = 0$, $\partial_{i,j-1} \partial_{i,j} = 0$ and $\kappa_{i,j-1} \partial_{i,j} = \partial_{i-1,j} \kappa_{i,j}$.
Let $\tilde{\Omega}_\bullet$ be the \emph{total complex} of $\Omega_\bullet F_\bullet$;
i.e, the chain complex with components and boundary morphisms
\begin{equation}
	\label{eq:total-cplx}
	\tilde{\Omega}_q = \smashoperator{\bigoplus_{i+j = q}} \Omega_i F_j,\qquad
	\partial^{\tilde{\Omega}}_q = \smashoperator{\sum_{i+j=q}} \bigl((-1)^q \kappa_{ij} + \partial_{ij}\bigr).
\end{equation}
\begin{example}
	\label{ex:total-cplx}
	Let $n = 2$ and $F_\bullet$ be a chain complex of free modules of the form $F_\bullet = (F_2 \xto{\partial_2} F_1 \xto{\partial_1} F_0)$.
	Then $\tilde{\Omega}_\bullet$ is the chain complex
	\begin{equation}
		\label{ex:total-cplx:block-matrices}
		\Omega_2 F_2
		\xto{\Mtx{\partial_{22} \\ \kappa_{22}}}
		\begin{matrix*}[r]
			\Omega_2 F_1 \\ \oplus \Omega_1 F_2
		\end{matrix*}
		\xto{\Mtx{\partial_{21} & 0 \\ -\kappa_{21} & \partial_{12} \\ 0 & -\kappa_{12}}}
		\begin{matrix*}[r]
			\Omega_2 F_0 \\ \oplus \Omega_1 F_1 \\ \oplus \Omega_0 F_2
		\end{matrix*}
		\xto{\Mtx{\kappa_{20} & \partial_{11} & 0 \\ 0 & \kappa_{11} & \partial_{02}}}
		\begin{matrix*}[r]
			\Omega_1 F_0 \\ \oplus \Omega_0 F_1
		\end{matrix*}
		\xto{(-\kappa_{10}\: 0)}
		\Omega_0 F_0.
	\end{equation}
	To make this example more explicit, let $F_\bullet$ be the chain complex $F_\bullet = \bigl(F\tbinom11 \xto{\Mtx{1\\1}} F\tbinom10 \oplus F\tbinom01 \xto{(-1\ 1)} F\tbinom00\bigr)$.
	Then $\Omega_\bullet F_\bullet$ is the double complex
	\[
		\begin{tikzcd}[column sep={13em,between origins}, row sep={4em,between origins}, plus/.style={gray, "\oplus" {description, sloped}, -, dotted}, eq/.style={gray, equal, dotted},cells={font=\small}]
			                                                       & [-6.5em] F\tbinom11 \rar{\partial_{22}} \dar{\kappa_{22}}                                            & F\tbinom10 \oplus F\tbinom01 \rar{\partial_{21}} \dar{\kappa_{21}}                                                                                           & F\tbinom00 \dar{\kappa_{20}}                                     \\
			|[shift={(0,2em)}, gray]| \tilde{\Omega}_4 \ar[ur,eq]  & F\tbinom{1}{-\infty} \oplus F\tbinom{-\infty}{1} \rar{\partial_{12}} \dar{\kappa_{12}} \ar[ur, plus] & F\tbinom{1}{-\infty} \oplus F\tbinom{0}{-\infty} \oplus F\tbinom{-\infty}{0} \oplus F\tbinom{-\infty}{1} \rar{\partial_{11}} \dar{\kappa_{11}} \ar[ur, plus] & F\tbinom{0}{-\infty}\oplus F\tbinom{-\infty}{0}\dar{\kappa_{10}} \\
			|[shift={(0,2em)}, gray]| \tilde{\Omega}_3 \ar[ur,eq]  & F\tbinom{-\infty}{-\infty} \rar{\partial_{02}} \ar[ur, plus]                                         & F\tbinom{-\infty}{-\infty} \oplus F\tbinom{-\infty}{-\infty} \rar{\partial_{01}} \ar[ur, plus]                                                               & F\tbinom{-\infty}{-\infty}                                       \\[-2em]
			|[shift={(0,   0)}, gray]| \tilde{\Omega}_2 \ar[ur,eq] & |[shift={(6.5em,0)}, gray]|\tilde{\Omega}_1 \ar[ur, eq]                                              & |[shift={(6.5em,0)}, gray]| \tilde{\Omega}_0 \ar[ur, eq]                                                                                                     &
		\end{tikzcd}
	\]
	The parts $\tilde{\Omega}_k$ of its total complex $\tilde{\Omega}_\bullet$ are the direct sums of the modules
	on the same (dotted) diagonal.
	The underlying ungraded matrices of the morphisms $\kappa_{ij}$ and $\partial_{ij}$ are
	\begin{align*}
		\u{\partial_{22}} = \u{\partial_{02}} &= \Mtx{1 \\ 1}, &
		\u{\kappa_{22}} = \u{\kappa_{20}} & = \Mtx{1 \\ 1},\\
		\u{\partial_{22}} = \u{\partial_{02}} &= (-1\ 1), &
		\u{\kappa_{12}} = \u{\kappa_{10}} &= (-1\ 1),\\
		\u{\partial_{12}} &= \begin{psmallmatrix}
			1 & 0 \\
			1 & 0 \\
			0 & 1 \\
			0 & 1
		\end{psmallmatrix},&
		\u{\kappa_{21}} &= \begin{psmallmatrix}
			1 & 0 \\
			0 & 1 \\
			1 & 0 \\
			0 & 1
		\end{psmallmatrix},\\
		\u{\partial_{11}} &= \begin{psmallmatrix*}[r]
			-1 &  1 &  0 & 0 \\
			 0 &  0 & -1 & 1
		\end{psmallmatrix*},&
		\u{\kappa_{11}} &= \begin{psmallmatrix*}[r]
			-1 &  0 & 1 & 0 \\
			 0 & -1 & 0 & 1
		\end{psmallmatrix*}.
	\end{align*}
	The boundary morphisms of $\tilde{\Omega}_\bullet$ are block matrices assembled from these
	as in \eqref{ex:total-cplx:block-matrices}.
\end{example}

The augmentation morphisms $\varepsilon_{F_i}\colon \Omega_\bullet F_i \stackrel\simeq\to \nu F_i$ from \eqref{eq:F-Omega-nuF}
induce a quasi-isomorphism $\tilde{\varepsilon}_\bullet\colon \tilde{\Omega}_\bullet \to \nu F_\bullet$ of chain complexes.
Similarly, the morphisms $\varpi_{F_i}\colon \Omega_\bullet F_i \to F_i[-n]$ from \eqref{eq:F-Omega-nuF} induce a morphism $\tilde{\varpi}_\bullet\colon \tilde{\Omega}_\bullet \to F_\bullet[-n]$ of chain complexes.
We obtain the diagram
\begin{equation}
	\label{eq:epsilon-pi}
	\begin{tikzcd}
		              & \tilde{\Omega}_\bullet \ar[dr, "\simeq", "\tilde{\varepsilon}_\bullet"'] \ar[dl, "\tilde{\varpi}_\bullet"] &  \\
		F_\bullet[-n] &                                                                                            & \nu F_\bullet,
	\end{tikzcd}
\end{equation}
where $\tilde{\varepsilon}$ is a quasi-isomorphism.

\begin{example}[Continuation of \zcref{ex:total-cplx}]
	\label{ex:total-cplx-contd}
	Recall the total complex $\tilde{\Omega}_\bullet$ from \eqref{ex:total-cplx:block-matrices}.
	In this situation, \eqref{eq:epsilon-pi} takes the form
	\[
				\def\X{\vphantom{\ensuremath{\Omega_0 F_0}}}
		\begin{tikzcd}[ampersand replacement=\&, column sep=4.25em]
			F_\bullet[-2]\colon \&[-5em] F_2 \rar{\partial_2} \&[-2.5em] F_1 \rar{\partial_1} \&[1em] F_0 \rar \&[1em] 0 \rar \&[-1.5em] 0 \\
			\tilde{\Omega}_\bullet\colon \uar{\tilde{\varpi}_\bullet} \dar[swap]{\tilde{\varepsilon}_\bullet} \&
			\begin{matrix*}[r]
				\Omega_2 F_2
			\end{matrix*}
			\rar{\Mtx{\partial_{22} \\ \kappa_{22}}}
			\uar{\id_{F_2}}
			\dar \&
			\begin{matrix*}[r]
				\Omega_2 F_1 \\ {}\oplus \Omega_1 F_2
			\end{matrix*}
			\rar{\Mtx{\partial_{21} & 0 \\ -\kappa_{21} & \partial_{12} \\ 0 & -\kappa_{12}}}
			\uar{(\id_{F_1} 0)}
			\dar \&
			\begin{matrix*}[r]
				\Omega_2 F_0 \\ {}\oplus \Omega_1 F_1 \\ {}\oplus \Omega_0 F_2
			\end{matrix*}
			\rar{\Mtx{\kappa_{20} & \partial_{11} & 0 \\ 0 & \kappa_{11} & \partial_{02}}}
			\uar[swap]{(\id_{F_0}\ 0\ 0)}
			\dar{(0\ 0\ \varepsilon_{F_2})} \&
			\begin{matrix*}[r]
				\Omega_1 F_0 \\ {}\oplus \Omega_0 F_1
			\end{matrix*}
			\rar{(-\kappa_{10}\: 0)}
			\uar
			\dar{(0\ \varepsilon_{F_1})} \&
			\begin{matrix*}[r]
				\Omega_0 F_0
			\end{matrix*}
			\uar
			\dar{\varepsilon_{F_0}} \\
			\nu F_\bullet\colon \& 0 \rar \& 0 \rar \& \nu F_2 \rar[swap]{\nu \partial_2} \& \nu F_1 \rar[swap]{\nu \partial_1} \& \nu F_0.
		\end{tikzcd}
	\]
\end{example}

We are now ready to state our central result:
\begin{theorem}
	\label{thm:main-preview}
	If the chain complex $F_\bullet$ is eventually acyclic, then $\tilde{\varpi}_\bullet$ is a homotopy equivalence.
\end{theorem}

We will establish \zcref{thm:main-preview} by providing a homotopy inverse $\tilde{\varpi}'_\bullet$ of $\tilde{\varpi}_\bullet$; see \zcref{thm:main}.
Before, we give a simple example of the theorem:

\begin{example}
	Let $n = 1$, and consider the free resolution $F_\bullet\colon F(1) \xto{1} F(0)$ of the simple module $\SimpleModule$,
	where $\SimpleModule_z = \F$ is $z = 0$ and $\SimpleModule_z = 0$ otherwise.
	Then $\tilde{\varpi}_\bullet$ fits into the diagram
	\[
		\begin{tikzcd}
			\tilde{\Omega}_\bullet\colon \dar[shift left=3pt, "\tilde{\varpi}_\bullet"] & F(1) \rar["\Mtx{1\\1}"'] \dar[shift left=3pt, "1"] & F(0)\oplus F(-\infty) \dar[shift left=3pt, "(1\ 0)"] \rar["(-1\ 1)"'] \lar[dashed, bend right, "0"'] & F(-\infty) \dar[shift left=3pt] \lar[dashed, bend right, "\Mtx{0\\1}"'] \\
			F_\bullet[-1]\colon \uar[shift left=3pt, "\tilde{\varpi}'_\bullet"]         & F(1) \rar["1"'] \uar[shift left=3pt, "1"]          & F(0) \uar["\Mtx{1\\1}", shift left=3pt] \rar                                                         & 0 \uar[shift left=3pt]
		\end{tikzcd}
	\]
	with the right term in degree zero.
	The morphism $\tilde{\varpi}'_\bullet$ is a homotopy inverse of $\tilde{\varpi}_\bullet$, exhibited by the dashed arrows.
\end{example}

\section{Constructing the quasi-inverse \texorpdfstring{$\tilde{\varpi}_\bullet'$}{ϖ̃'}}
The goal for this section is to show that the morphism $\tilde{\varpi}_\bullet$ from \eqref{eq:epsilon-pi} is a homotopy equivalence
by providing a homotopy inverse $\tilde{\varpi}'_\bullet$.
As before, let $F_\bullet$ be an eventually acyclic chain complex of finite rank free modules in $\Pers{n}$.
By \zcref{thm:eventually-acyclic}, the chain complex $\colim_Q F_\bullet$ is acyclic for every $Q$.
It is a complex of free $(n-\abs{Q})$-parameter modules, so by \zcref{rmk:obtaining-contraction}, it is contractible.
Let $\bar{s}^Q_\bullet$ be a chain contraction of $\colim_Q F_\bullet$.
Then also the chain complex $\Colim_Q F_\bullet$ of flat $n$-parameter modules is contractible, with
the chain contraction $s^Q_\bullet \coloneqq \Delta_Q \bar{s}^Q_\bullet$.
Hence, for every $0 \leq i < n$, the chain complex $\Omega_i F_\bullet = \bigoplus_{Q \in \binom{\IntSet{n}}{n - i}} \Colim_Q F_\bullet$ is contractible 
with the chain contraction $s_{i\bullet} \coloneqq \bigoplus_{Q \in \binom{\IntSet{n}}{n - i}} s^Q_\bullet$.
Explicitly, this means that for all $j \in \Z$ and $0 \leq i < n$,
there are morphisms $s_{ij}\colon \Omega_i F_j \to \Omega_i F_{j+1}$ such that
\begin{equation}
	\label{eq:Koszul stage contraction}
	\partial_{i,j+1} s_{ij} + s_{i,j-1} \partial_{ij} = \id_{\Omega_i F_j}.
\end{equation}
We will use these contractions to build the desired morphism $\tilde{\varpi}'_\bullet$.
Recall the boundary morphism $\kappa_{ij}\colon \Omega_i F_j \to \Omega_{i-1} F_j$ from \eqref{eq:kappa-partial}.
For $k \geq 0$, we recursively define morphisms $t_{ijk}\colon \Omega_i F_j \to \Omega_{i-k} F_{j+k}$ by
\begin{equation}
	\label{eq:flat-injective presentations:t}
	\begin{split}
		t_{ij0} &= \id_{\Omega_i F_j},\\
		t_{i,j,k+1} &= s_{i-k-1, j+k} \, \kappa_{i-k, j+k} \, t_{i,j,k}
	\end{split}
\end{equation}
We obtain the diagram
\begin{equation}
	\label{eq:big-cd}
	\begin{tikzcd}[s/.style={densely dashed, bend right, swap}, t/.style={dotted, near end}, column sep={6.25em,between origins}]
		       &[-3em] \vdots \dar                                                                      & \vdots\dar                                                                               & \iddots               & \vdots \dar                                                                                        & \vdots \dar                                                                                                & \vdots       \dar                                                \\
		0 \rar & \Omega_n F_2 \rar["\kappa_{n2}"] \dar["\partial_{n2}"' near end] \ar[ur, t]            & \Omega_{n-1} F_2 \rar \dar["\partial_{n-1,2}"' near end] \uar[s]              \ar[ur, t] & \cdots \rar\ar[ur, t] & \Omega_2 F_2 \rar["\kappa_{22}"] \dar["\partial_{22}"' near end] \uar[s]           \ar[ur, t]            & \Omega_1 F_2 \rar["\kappa_{12}"] \dar["\partial_{12}"' near end]  \uar[s ]           \ar[ur, t]            & \Omega_0 F_2 \dar["\partial_{02}"' near end]   \uar[s]           \\
		0 \rar & \Omega_n F_1 \rar["\kappa_{n1}"] \dar["\partial_{n1}"' near end] \ar[ur, t, "t_{n11}"] & \Omega_{n-1} F_1 \rar \dar["\partial_{n-1,1}"' near end] \uar[s, "s_{n-1,1}"] \ar[ur, t] & \cdots \rar\ar[ur, t] & \Omega_2 F_1 \rar["\kappa_{21}"] \dar["\partial_{21}"' near end] \uar[s, "s_{21}"] \ar[ur, t, "t_{211}"] & \Omega_1 F_1 \rar["\kappa_{11}"] \dar["\partial_{11}"' near end]  \uar[s, "s_{11}" ] \ar[ur, t, "t_{111}"] & \Omega_0 F_1 \dar["\partial_{01}"' near end]   \uar[s, "s_{01}"] \\
		0 \rar & \Omega_n F_0 \rar["\kappa_{n0}"]                                 \ar[ur, t, "t_{n01}"] & \Omega_{n-1} F_0 \rar                                    \uar[s, "s_{n-1,0}"] \ar[ur, t] & \cdots \rar\ar[ur, t] & \Omega_2 F_0 \rar["\kappa_{20}"]                                 \uar[s, "s_{20}"] \ar[ur, t, "t_{201}"] & \Omega_1 F_0 \rar["\kappa_{10}"]                                  \uar[s, "s_{10}" ] \ar[ur, t, "t_{101}"] & \Omega_0 F_0                                   \uar[s, "s_{00}"]
	\end{tikzcd}
\end{equation}
of which the solid arrows commute.
The dashed arrows form the contractions $s_{i\bullet}$ of the complexes $\Omega_i F_\bullet$.
The dotted arrows are the morphisms $t_{ij1}$.
The morphisms $t_{ijk}$ for $k > 1$ are the compositions of consecutive diagonal arrows $t_{ij1}$.
In the leftmost column $\Omega_\bullet F_\bullet = F_\bullet$, there is no contraction $s_{n\bullet}$.

We are now ready to establish our main theorem, which gives an explicit construction for the quasi-inverse $\tilde{\varpi}'_\bullet$:
\begin{theorem}
	\label{thm:main}
	Let $F_\bullet$ be an eventually acyclic chain complex of free modules, and let $\tilde{\Omega}_\bullet$, $\tilde{\varpi}_\bullet$, and $t_{ijk}$ be as above.
	\begin{enumerate}
		\item\label{thm:main:1} The morphisms
		\settowidth\ScratchLen{$(-1)^{n(q+1)}  t_{nqn}$}
		\begin{equation*}
			\tilde{\varpi}'_q \coloneqq
			\begin{psmallmatrix}
				\id_F \\
				(-1)^{q+1} t_{nq1} \\
				\phantom{(-1)^{q+1}} t_{nq2} \\
				(-1)^{q+1} t_{nq3} \\
				\phantom{(-1)^{q+1}} t_{nq4} \\
				\phantom{(-1)^{q+1}} \hspace{0.5ex} \vdots \\
				(-1)^{n(q+1)}  t_{nqn}
			\end{psmallmatrix}
			\colon F_q \to \underbrace{\Omega_n F_q \oplus \dotsb \oplus \Omega_0 F_{q+n}}_{\tilde{\Omega}_{q+n}},
		\end{equation*}
		of modules form a morphism $\tilde{\varpi}'_\bullet \colon F_\bullet \to \tilde{\Omega}_\bullet\HShift{n}$ of chain complexes.
		\item \label{thm:main:2} It is a homotopy inverse of $\tilde{\varpi}_\bullet\colon \tilde{\Omega}_\bullet\HShift{n} \to F_\bullet$.
		\item \label{thm:main:3} The composition $\phi_\bullet\colon F_\bullet \xto{\varpi_\bullet'} \tilde{\Omega}_\bullet[n] \xto{\tilde{\varepsilon}_\bullet} \nu F_\bullet[n]$ is a quasi-isomorphism.
	\end{enumerate}
\end{theorem}

Before we give the proof, we show a small example.
\begin{example}[Continuation of \zcref{ex:total-cplx-contd}]
	Recall the morphism $\tilde{\varpi}_\bullet\colon \tilde{\Omega}_\bullet \to F_\bullet[-2]$ of chain complexes from \zcref{ex:total-cplx-contd}.
	In this situation, $\tilde{\varpi}_\bullet$ and $\tilde{\varpi}'_\bullet$ are the morphisms
	\[
		\def\X{\vphantom{\ensuremath{\Omega_0 F_0}}}
		\begin{tikzcd}[ampersand replacement=\&, row sep=large, column sep=2.25em]
			F_\bullet\colon \dar["\tilde{\varpi}'_\bullet", shift left]\&[-1.5em]
			F_2 \dar[shift left]{\id_{F_2}} \rar{\partial_2} \&[-.5em]
			F_1 \dar[shift left]{\Mtx{\id_{F_1} \\ s_{11} \kappa_{21}}} \rar{\partial_1} \&[2.5em]
			F_0 \dar[shift left]{\Mtx{\id_{F_0} \\ -s_{10} \kappa_{20} \\ s_{01} \kappa_{11} s_{10} \kappa_{20}}} \rar \&[2.5em]
			0 \dar[shift left] \rar \&[.5em] 
			0 \dar[shift left]
			\\
			\tilde{\Omega}_\bullet[2]\colon \uar[shift left, "\tilde{\varpi}_\bullet"]\&
			\Omega_2 F_2
			\uar[shift left, "\id_{F_2}"]
			\rar[swap]{\Mtx{\partial_{22} \\ \kappa_{22}}} \&
				\begin{matrix*}[r]
					\Omega_2 F_1 \\ {} \oplus \Omega_1 F_2
				\end{matrix*}
			\uar[shift left, "{(\id_{F_1}\ 0)}"]
			\rar[swap]{\Mtx{\partial_{21} & 0 \\ -\kappa_{21} & \partial_{12} \\ 0 & -\kappa_{12}}} \&
				\begin{matrix*}[r]
					\Omega_2 F_0 \\  {} \oplus\Omega_1 F_1 \\ {} \oplus \Omega_0 F_2
				\end{matrix*}
			\uar[shift left, "({\id_{F_0}\ 0\ 0})"]
			\rar[swap]{\Mtx{\kappa_{20} & \partial_{11} & 0 \\ 0 & \kappa_{11} & \partial_{02}}} \&
				\begin{matrix*}[r]
					\Omega_1 F_0 \\ \oplus \Omega_0 F_1
				\end{matrix*}
			\uar[shift left] \rar[swap]{(-\kappa_{10, 0})} \&
			\Omega_0 F_0,
			\uar[shift left]
		\end{tikzcd}
	\]
	with the middle term in homological degree zero.
	To see that $\tilde{\varpi}'_\bullet$ is a morphism of chain complexes, we have to show that the squares with downward morphisms commute.
	Since $s_{i\bullet}$ is a contraction for each $i$, we have
	\begin{align}
		\label{4QWOch3Ociss}
		\id_{\Omega_k F_2} &= s_{k1} \partial_{k2} + \underbrace{\partial_{k3} s_{k2}}_{0} \\
		\label{wWaDxdhCCbd3}
		\id_{\Omega_k F_1} &= s_{k0} \partial_{k1} + \partial_{k2} s_{k1}
	\end{align}
	for all $k$.
	For the first square, we get that
	\[
		\begin{pmatrix}\id_{\Omega_2 F_1} \\ s_{11} \kappa_{21}\end{pmatrix} \partial_2
		= \begin{pmatrix}\partial_{22} \\ s_{11} \kappa_{21} \partial_{22} \end{pmatrix}
		= \begin{pmatrix}\partial_{22} \\ s_{11} \partial_{12} \kappa_{22} \end{pmatrix}
		\stackrel{\eqref{4QWOch3Ociss}}{=} \begin{pmatrix} \partial_{22} \\ \kappa_{22} \end{pmatrix}
		= \begin{pmatrix}\partial_{22} \\ \kappa_{22} \end{pmatrix}  \id_{\Omega_2 F_2},
	\]
	so the first square commutes.
	For the second square, we get
	\begin{multline*}
		  \begin{pmatrix} \id_{\Omega_2 F_0} \\ -s_{10} \kappa_{20} \\ s_{01} \kappa_{11} s_{10} \kappa_{20}\end{pmatrix} \partial_1
		= \begin{pmatrix} \partial_{21} \\ -s_{10} \kappa_{20} \partial_{21} \\ s_{01} \kappa_{11} s_{10} \kappa_{20} \partial_{21}\end{pmatrix}\\
		= \begin{pmatrix} \partial_{21} \\ -\kappa_{21} + \partial_{12} s_{11} \kappa_{21}  \\ - \kappa_{12} s_{11} \kappa_{21}\end{pmatrix}
		= \begin{pmatrix} \partial_{21} & 0 \\ -\kappa_{21} & \partial_{12} \\ 0 & -\kappa_{12}\end{pmatrix}
		  \begin{pmatrix} \id_{\Omega_2 F_1} \\ s_{11} \kappa_{21}\end{pmatrix}
	\end{multline*}
	because
	\begin{gather*}
		-s_{10} \kappa_{20} \partial_{21}
			= -s_{10} \partial_{11} \kappa_{21}
			= -\kappa_{21} + (\id_{\Omega_1 F_1} - s_{10} \partial_{11}) \kappa_{21}
			\stackrel{\eqref{wWaDxdhCCbd3}}{=} -\kappa_{21} + \partial_{12} s_{11} \kappa_{21},\\\intertext{and}
		\begin{multlined}[\linewidth-2cm]
			s_{01} \kappa_{11} s_{10} \kappa_{20} \partial_{21}
			= s_{01} \kappa_{11} s_{10} \partial_{11} \kappa_{21}
			\stackrel{\eqref{wWaDxdhCCbd3}}{=} -s_{01} \kappa_{11} \partial_{12} s_{11} \kappa_{21}\\
			= -s_{01} \partial_{02} \kappa_{12} s_{11} \kappa_{21}
			\stackrel{\eqref{4QWOch3Ociss}}{=} \kappa_{12} s_{11} \kappa_{21}.
		\end{multlined}
	\end{gather*}
	For the third square, we check
	\[
		  \begin{pmatrix}\kappa_{20} \partial_{11} & 0 & 0 \\ 0 & \kappa_{11} & \partial_{02}\end{pmatrix}
		  \begin{pmatrix}\id \\ -s_{10} \kappa_{20} \\ s_{01} \kappa_{11} s_{10} \kappa_{20}\end{pmatrix}
		= \begin{pmatrix}\kappa_{20} - \partial_{11} s_{10} \kappa_{20} \\ \kappa_{11} s_{10} \kappa_{20} + \partial_{02} s_{01} \kappa_{11} s_{10} \kappa_{20}\end{pmatrix}
		= 0
	\]
	because
	\begin{gather*}
		\kappa_{20} - \partial_{11} s_{10} \kappa_{20} \stackrel{\eqref{4QWOch3Ociss}}{=} \kappa_{20} - \kappa_{20} = 0,\\
		\kappa_{11} s_{10} \kappa_{20} + \underbrace{\partial_{02} s_{01}}_{\id - s_{00} \partial_{01}} \kappa_{11} s_{10} \kappa_{20}
			= - s_{00} \partial_{01} \kappa_{11} s_{10} \kappa_{20}
			= - s_{00} \kappa_{10} \underbrace{\partial_{11} s_{10}}_{\id} \kappa_{20}
			= 0.
	\end{gather*}
	One can show in a similar way that $\tilde{\varpi}_\bullet$ and $\tilde{\varpi}'_\bullet$ are chain homotopy inverses.
\end{example}
\begin{proof}[Proof of \zcref{thm:main}]
	For simplicity of notation,
	when writing down a composition of morphisms $\kappa_{ij}$, $\partial_{ij}$, $s_{ij}$ and $t_{ijk}$, we omit the indices $i$ and $j$ from all but the rightmost morphism.
	For the others, these indices then are determined because each of these has domain $\Omega_i F_j$.
	For example, in the following claim, $\kappa \partial t_{ijk}$ stands for $\kappa_{i-k,j+k-1} \, \partial_{i-k,j+k} \, t_{ijk}$,
	and $\kappa t_k \partial_{ij}$ stands for $\kappa_{i-k,j+k-1}\, t_{i,j-1,k} \, \partial_{ij}$.
	\begin{claim*}
		For all $i < n$, $j$ and $k$, we have
		\begin{equation}
			\label{eq:pulling partial over t}
			\kappa \partial t_{ijk} = (-1)^k \kappa t_k \partial_{ij}.
		\end{equation}
	\end{claim*}
	\begin{claimproof}
		We show this by induction over $k$.
		Applying the induction hypothesis at $(*)$
		and using that $\partial_{\bullet\bullet}$ and $\kappa_{\bullet\bullet}$ form a double complex in $(\dagger)$,
		 we get
		\begin{multline*}
			\kappa \partial t_{ijk}
			\stackrel{\eqref{eq:flat-injective presentations:t}}{=}
			\kappa \partial s \kappa t_{i,j,k-1}
			\stackrel{\eqref{eq:Koszul stage contraction}}{=}
			\kappa (1 -  s \partial ) \kappa t_{i,j,k-1}
			\stackrel{(\dagger)}{=}
			- \kappa s \partial \kappa t_{i,j,k-1} \\
			\stackrel{(\dagger)}{=}
			- \kappa s \kappa \partial t_{i,j,k-1}
			\stackrel{(*)}{=}
			(-1)^k \kappa s \kappa t_{k-1} \partial_{ij}
			\stackrel{\eqref{eq:flat-injective presentations:t}}{=}
			(-1)^k \kappa t_k \partial_{ij}.\qedhere
		\end{multline*}
	\end{claimproof}
	The claim implies that 	for all $i < n$, $j$ and $k$, we have
	\begin{multline}
		\label{eq:t contraction relation}
		\partial t_{ijk} - \kappa t_{i,j,k-1}
		\stackrel{\eqref{eq:flat-injective presentations:t}}{=} (\partial s - 1) \kappa t_{i,j,k-1}
		\stackrel{\eqref{eq:Koszul stage contraction}}{=} -s \partial \kappa t_{i,j,k-1}\\
		\stackrel{\eqref{eq:pulling partial over t}}{=} -(-1)^{k-1} s \kappa t_{k-1} \partial_{ij}
		\stackrel{\eqref{eq:flat-injective presentations:t}}{=} (-1)^{k} t_k \partial_{ij}.
	\end{multline}

	To show \proofref{thm:main:1}, recall the boundary morphism $\partial^{\tilde{\Omega}}_q$ of $\tilde{\Omega}_\bullet$ from \eqref{eq:total-cplx}.
	Let $\pm$ stand for $(-1)^q$ and $\mp$ for $(-1)^{q+1}$.
	We obtain that
	\begin{multline*}
		\partial^{\tilde{\Omega}}_q \tilde{\varpi}'_q
		=
		\scriptstyle
		\setlength{\arraycolsep}{2pt}
		\begin{pNiceMatrix}[]
			\partial_{nq}   & 0                    & \Cdots                 &                      & 0                     \\
			\pm \kappa_{nq} & \partial_{n-1,q+1}   & \Ddots                 &                      & \Vdots                \\
			0               & \pm \kappa_{n-1,q+1} & \;\partial_{n-2,q+2}\; &                      &                       \\
			\Vdots          & \Ddots               & \Ddots                 & \Ddots[draw-first]   & 0                     \\
			0               & \Cdots               & 0                      & \pm \kappa_{1,q+n-1} & \;\partial_{0, q+n}\;
		\end{pNiceMatrix}
		\begin{pmatrix}
			\id_{\Omega_n F_q}       \\
			\mp          t_{nq1}     \\
			             t_{nq2}     \\
			\vdots                   \\
			(\mp)^{n-1}  t_{n,q,n-1} \\
			(\mp)^{n}    t_{nqn}
		\end{pmatrix}
		\\
		=
		\begin{pmatrix}
			            \partial_{nq}                          \\
			\mp        (\partial t_{nq1} - \kappa_{nq})        \\
			-          (\partial t_{nq2} - \kappa t_{nq1})     \\
			\vdots                                             \\
			(\mp)^{n}  (\partial t_{nqn} - \kappa t_{n,q,n-1})
		\end{pmatrix}
		\stackrel{\eqref{eq:t contraction relation}}{=}
		\begin{pmatrix}
			              \partial_{nq} \\
			\pm       t_1 \partial_{nq} \\
			          t_2 \partial_{nq} \\
			\vdots                   \\
			(\pm)^{n} t_n \partial_{nq}
		\end{pmatrix}
		= \tilde{\varpi}'_{q-1} \partial^{F}_q.
	\end{multline*}
	Therefore, $\tilde{\varpi}'_\bullet$ is a morphism of chain complexes.
	
	For \proofref{thm:main:2}, we have to show that $\tilde{\varpi}'_\bullet$ is a homotopy inverse of $\tilde{\varpi}_\bullet$.
	We clearly have $\tilde{\varpi}_\bullet \tilde{\varpi}'_\bullet = \id_{F_\bullet}$.
	For the other composition $\tilde{\varpi}'_\bullet \tilde{\varpi}_\bullet$, we will show that the morphism
	\begin{equation*}
		\sigma_j \coloneqq
		\begin{pNiceArray}{c r@{}l r@{}l r@{}l c c}
			0      & \multicolumn{2}{c}{0}      & \Cdots  &                  &         &                  &             & 0                                \\
			0      &         & s_{n-1,q+1}      &         &                  &         &                  &             & \Vdots                           \\
			0      & \mp t_1 & s_{n-1,q+1}      &         & s_{n-2,q+2}      &         &                  &             &                                  \\
			0      & t_2     & s_{n-1,q+1}      & \mp t_1 & s_{n-2,q+2}      &         & s_{n-3,q+3}      &             &                                  \\
			0      & \mp t_3 & s_{n-1,q+1}      & t_2     & s_{n-2,q+2}      & \mp t_1 & s_{n-3,q+3}      & s_{n-4,q+4} & 0                                \\
			\vdots & \multicolumn{2}{c}{\vdots} & \multicolumn{2}{c}{\vdots} & \multicolumn{2}{c}{\vdots} & \vdots      & \smash{\rotatebox{20}{$\ddots$}}
			\CodeAfter
			\line{1-2}{5-9}
		\end{pNiceArray}
		\colon
		\tilde{\Omega}_{q+n}
		\to
		\tilde{\Omega}_{q+n+1}.
	\end{equation*}
	form a homotopy from $\id_{\tilde{\Omega}_\bullet}$ to $\tilde{\varpi}'_\bullet \tilde{\varpi}_\bullet$.
	Note that
	\begin{equation}
		\label{eq:t two-sided recursion}
		t_{i,j,k+1}
		\stackrel{\eqref{eq:flat-injective presentations:t}}{=}
		s \kappa t_{ijk}
		= t_{k} s \kappa_{ij},
	\end{equation}
	which both follow directly from the definition of $t_{ijk}$.
	From this, it follows that for all $k \geq 0$, we have
	\settowidth\ScratchLen{$\scriptstyle \eqref{eq:t two-sided recursion}$}
	\newcommand\eq[1][]{\stackrel{\smash{\mathmakebox[\ScratchLen][c]{#1}}}{=}}
	\begin{align}
		\notag
		&\phantom{{}\eq{}}
		(\kappa t_k s_{ij} - \partial t_{k} s_{ij}) + (-1)^{k+1} (t_{k+1} s \partial_{ij} - t_k s \kappa_{ij}) \\
		\notag
		&\eq[\eqref{eq:t two-sided recursion}]
		(1 - \partial s) \kappa t_k s_{ij} + (-1)^{k} t_k s \kappa (s \partial_{ij} - 1) \\
		\notag
		&\eq[\eqref{eq:Koszul stage contraction}]
		s \partial \kappa t_k s_{ij} - (-1)^{k} t_k s \kappa \partial s_{ij} \\
		\notag
		&\eq[\eqref{eq:t two-sided recursion}]
		s \partial \kappa t_k s_{ij} - (-1)^{k} s \kappa t_k \partial s_{ij} \\
		\notag
		&\eq[\eqref{eq:pulling partial over t}]
		s \partial \kappa t_k s_{ij} - s \partial \kappa t_k s_{ij} \\
		\label{eq:contraction zero fields}
		&\eq
		0.
	\end{align}
	We are now ready to show that $\sigma_\bullet$ is the desired chain homotopy.
	In the following calculation, for the sake of legibility, we leave out the indices $i$ and $j$ of the morphisms $s_{ij}$, $\kappa_{ij}$, $\partial_{ij}$ and $t_{ijk}$
	We obtain
	\settowidth\ScratchLen{$\scriptstyle\eqref{eq:contraction zero fields}$}
	\begin{align*}
		&\phantom{{}\eq{}} \partial^{\tilde{\Omega}}_{j+1} \sigma_j + \sigma_{j-1} \partial^{\tilde{\Omega}}_j\\
		&\eq
		\begin{psmallmatrix}
			\partial \\
			\mp \kappa & \partial \\
			& \mp \kappa & \partial\\
			& & \mp \kappa & \partial\\
			& & & & \ddots
		\end{psmallmatrix}
		\begin{psmallmatrix}
			\mathstrut 0 \\
			0 & s \\
			0 & \pm t_1 s & s \\
			0 & t_2 s & \pm t_1 s & s \\
			\vdots & & & & \ddots
		\end{psmallmatrix}
		+
		\begin{psmallmatrix}
			\mathstrut 0 \\
			0 & s \\
			0 & \mp t_1 s & s \\
			0 & t_2 s & \mp t_1 s & s \\
			\vdots & & & & \ddots
		\end{psmallmatrix}
		\begin{psmallmatrix}
			\partial \\
			\pm \kappa & \partial \\
			& \pm \kappa & \partial\\
			& & \pm \kappa & \partial\\
			& & & & \ddots
		\end{psmallmatrix}
		\\
		&\eq
		\begin{pmatrix}
			0
			\\
			\phantom{t_0} {\pm s \kappa} &
			s \partial + \partial s
			\\
			- t_1 s \kappa &
			\mp \kappa s \pm \partial t_1 s \mp t_1 s \partial \pm s \kappa &
			s \partial + \partial s
			\\
			\pm t_2 s \kappa &
			-\kappa t_1 s + \partial t_2 s - t_1 s \kappa &
			\mp \kappa s \pm \partial t_1 s \mp t_1 s \partial \pm s \kappa &
			s \partial + \partial s \\
			\vdots & & & & \smash{\rotatebox{30}{$\ddots$}}
		\end{pmatrix}
		\\
		&\eq[\eqref{eq:contraction zero fields}] \mleft(
		\begin{array}{cc@{\hspace{-0.75ex}}c@{\hspace{-0.75ex}}c@{\hspace{-0.75ex}}c}
			0 \\
			\pm t_1 & \id_{\Omega_{n-1} F_{q+1}} \\
			-t_2 & 0 & \id_{\Omega_{n-2} F_{q+2}} \\
			\pm t_3 & 0 & 0 & \id_{\Omega_{n-3} F_{q+3}} \\
			\vdots & & & & \smash{\rotatebox{20}{$\ddots$}}
		\end{array}
		\mright) \\
		&\eq \id_{\tilde{\Omega}_{n+q}} - \tilde{\varpi}'_j \tilde{\varpi}_j.
	\end{align*}
	This shows that $\tilde{\varpi}'_\bullet \tilde{\varpi}_\bullet$ and $\id_{\tilde{\Omega}_\bullet}$ are chain homotopic.
	
	Lastly, \proofref{thm:main:3} follows from \proofref{thm:main:2} and the fact that $\tilde{\varepsilon}_\bullet$ is the quasi-isomorphism from \eqref{eq:epsilon-pi}.
\end{proof}

\section{Computing flat-injective presentations}
Assume that $F_\bullet$ is a free resolution of some finite dimensional module $M$.
We will use \zcref{thm:main} to construct a flat-injective presentation matrix of $M$.
Recall the augmentation morphism $\varepsilon_{F_n}\colon \Omega_0 F_n \to \nu F_n$
and the morphism $t_{n0n}\colon \Omega_n F_0 \to \Omega_0 F_n$ from \zcref{eq:F-Omega-nuF,eq:flat-injective presentations:t}.

\begin{corollary}
	\label{thm:flat-injective presentation module}
	Let $M \in \Pers{\Z^n}$ be finitely dimensional, $F_\bullet$ be a free resolution of $M$ of length $n$, and let $t_{ijk}$ be as above.
	Then the composite morphism
	\begin{equation*}
		\varphi\colon\quad F_0 \stackrel=\longrightarrow \Omega_n F_0 \xto{t_{n0n}} \Omega_0 F_n \stackrel{\varepsilon_{F_n}}{\longrightarrow} \nu F_n
	\end{equation*}
	is a flat-injective presentation of $M$, where
	\begin{equation}
		\label{eq:t_n0n}
		t_{n0n} = s_{0,n-1} \kappa_{1,n-1} \dotsm s_{n-1,0} \kappa_{n,0}.
	\end{equation}
\end{corollary}
\begin{proof}
	According to \zcref{thm:main},
	we have quasi-isomorphisms
	\begin{equation*}
		\begin{tikzcd}
			F_\bullet\mathrlap{\colon} \dar["\simeq", "\tilde{\varpi}'_\bullet"'] \ar[dd, "\phi_\bullet"', "\simeq", out=190, in=170] & \cdots \rar & F_1 \dar \rar                  & F_0 \dar \rar                & 0 \dar \rar                    & \cdots             \\
			\tilde{\Omega}_\bullet\HShift{n}\mathrlap{\colon} \dar["\simeq", "{\tilde{\varepsilon}_\bullet\HShift{n}}"']                 & \cdots \rar & \tilde{\Omega}_{n+1} \rar \dar & \tilde{\Omega}_{n} \rar \dar & \tilde{\Omega}_{n-1} \rar \dar & \cdots             \\
			I_\bullet\HShift{n}\mathrlap{\colon}                                                                                         & \cdots \rar & 0 \rar                         & I_n \rar                     & I_{n-1} \rar                   & \cdots\mathrlap{,}
		\end{tikzcd}
	\end{equation*}
	where $I_n = \nu F_n$, and $\tilde{\Omega}_{n} = \Omega_n F_{0} \oplus \dotsb \oplus \Omega_0 F_{n}$.
	The image of their composition $\phi_\bullet$ is quasi-isomorphic to $F_\bullet$ and hence to $M$,
	because $F$ is a free resolution of $M$.
	The only non-zero component of $\phi_\bullet$ is in degree zero:
	\begin{equation*}
		\varphi \coloneqq \phi_0 = \bigl(0, \dotsc, 0, \varepsilon_{F_n}\bigr)
		\begin{psmallmatrix}
			\id_{\Omega_n F_0} \\
			- t_{n01} \\
			\vdots \\
			(-1)^n t_{n0n}
		\end{psmallmatrix}
		= (-1)^n \varepsilon_{F_n} t_{n0n} \colon F_0 \to \Omega_0 F_n \to I_n.
	\end{equation*}
	This shows that $\im \varepsilon_{F_n} t_{n0n} = \im \tilde{\varepsilon}_n \tilde{\varpi}'_0 = M$.
\end{proof}

\begin{remark}
	Recall the diagram of the double complex $\Omega_\bullet F_\bullet$ from \eqref{eq:big-cd}.
	The morphism $\varphi$ is the composition
	\begin{equation}
		\label{eq:flat-injective-pres-decomposed}
		\begin{tikzcd}[row sep=small]
			                                           &                                  &             &                                             & \Omega_0 F_n \rar{\varepsilon_{F_n}} & \nu F_n, \\
			                                           &                                  &             & \Omega_1 F_{n-1} \rar[swap]{\kappa_{1,n-1}} & \Omega_0 F_{n-1} \uar{s_{0,n-1}}     &          \\
			                                           & \Omega_{n-1} F_1 \rar            & \cdots \rar & \Omega_1 F_{n-2} \uar["s_{1,n-2}"]          &                                      &          \\
			F_0 = \Omega_n F_0 \rar[swap]{\kappa_{n0}} & \Omega_{n-1} F_0 \uar{s_{n-1,0}} &             &                                             &                                      &
		\end{tikzcd}
	\end{equation}
	where all objects but $\nu F_n$ are flat, and $\nu F_n$ is injective.
\end{remark}

We finish this section by a short comment on minimal flat-injective presentations.
A flat-injective presentation $\varphi\colon F \to I$ of a module $M$ is called \emph{minimal}
if $F \to M$ and $I^* \to M^*$ are minimal generating systems of $M$ and $M^*$, respectively.
See \cite{GrimpenStefanou:2025}.

\begin{proposition}
	If $F_\bullet$ is a minimal free resolution of a finite dimensional module $M \in \Pers{\Z^n}$
	then the flat-injective presentation $\varphi$ from \zcref{thm:flat-injective presentation module} is minimal.
\end{proposition}
\begin{proof}
	If $F_\bullet$ is a minimal free resolution of $M$,
	then $F_0 \to M$ is a minimal generating system.
	Both dualities $(-)^*$ and $(-)^\dagger$ preserves trivial summands.
	In particular, $\nu$ maps minimal free to minimal cofree resolutions,
	so $I^\bullet \coloneqq \nu F_\bullet[n]$ is a minimal cofree resolution of $M$.
	It follows that $(I^\bullet)^*$ is a minimal free resolution of $M^*$, so $(I^0)^* = F_n^\dagger\Shift{-\one} \to M^*$ is a minimal generating system.
\end{proof}

\subsection{Representative matrix for \texorpdfstring{$\varphi$}{φ}}
\label{sec:flange-pres-matrix}
We now have a formula for the flat-injective presentation $\varphi$ of $M$ per \zcref{thm:flat-injective presentation module}.
In this section, we explain how to obtain a matrix $\varphi$ representing it.
Assume that we are given matrices representing the free resolution $F_\bullet$ of $M$ with respect to a fixed choice of bases.
We then construct the matrix $\varphi$
such that it represents $\varphi$ with respect to the same basis of $F_0$, and the dual basis of $\nu F_n$.

Before, we introduce some notation.
Recall the function $p_Q\colon \Z^n \to \Z^{n-\abs{Q}}$ for $Q \subseteq \IntSet{n}$ that forgets the coordinates indexed by $Q$.
It has a section $r_Q\colon \Z^{n-\abs{Q}} \to \lZ^{n}$ that inserts $-\infty$ for the coordinates indexed by $Q$.
For a $\Z^n$-graded matrix $U$, let $p_Q(U)$ and $r_Q(U)$ be the graded matrices obtained by applying $p_Q$ and $r_Q$, respectively,
to the row and column grades of $U$.
If $U$ represents a morphism $u$ of free modules in $\Pers{n}$, then $p_Q(U)$ represents $\colim_Q u$
with respect to the induced bases, and analogously for $\Delta_Q$ and $r_Q$.

\paragraph{Matrices \texorpdfstring{$S_{ij}$}{Sᵢⱼ} representing the contractions \texorpdfstring{$s_{ij}$}{sᵢⱼ}}
Let $M \in \Pers{n}$ be a finite dimensional module and $F_\bullet$ be a free resolution of length $n$ of $M$.
Let $D_1,\dotsc,D_n$ be graded matrices representing $F_\bullet$ with respect to some fixed basis.
Since $M$ is finite dimensional, the chain complex $\colim_Q F_\bullet$ is acyclic and hence contractible for each nonempty $Q$.
The proof of \zcref{rmk:obtaining-contraction} yields an explicit algorithm to compute from $D_1,\dotsc,D_n$
graded matrices $\bar{S}^Q_0,\dotsc,\bar{S}^Q_{n-1}$ that represent a chain contraction $\bar{s}^Q_\bullet$ of $\colim_Q F_\bullet$.
Now the matrices $S^Q_j \coloneqq r_Q (\bar{S}^Q_j)$
represent the chain contraction $s^Q_\bullet = \Delta_Q \bar{s}^Q_\bullet$ of $\Colim_Q F_\bullet$,
and for each $k < n$, the block matrices
\[
	S_{ij} \coloneqq \smashoperator{\bigoplus_{Q \in \binom{\IntSet{n}}{n-i}}} S^Q_j
\]
represent the desired chain contraction morphisms $s_{ij}$.

\begin{remark}
	\label{rmk:representative-index-Q}
	For each $Q$ with $\abs{Q} \geq 1$, we may choose an element $k_Q \in Q$
	and set $\bar{s}^Q_\bullet \coloneqq p_Q (r_{\{k_Q\}} (\bar{s}^{\{k_Q\}}_\bullet))$.
	This way, it suffices to compute only the $n$ contractions $s_\bullet^{\{i\}}$ for $1 \leq i \leq n$.
\end{remark}

\paragraph{Matrices $K_{ij}$ representing the boundary morphism $\kappa_{ij}$}
Recall that $\kappa_{ij} = \kappa_i \otimes F_j$, where $\kappa_i$ is the $i$th boundary morphism of the chain complex $\Omega_\bullet$.
From the construction of $\Omega_\bullet$, we see that the graded matrix $K_i$ representing $\kappa_i$
is the $i$th boundary matrix of the standard $n$-simplex, with the row and column grades dictated by $\Omega_\bullet$.

\begin{definition}
	For multisets $r$ and $s$, we let $r * s \coloneqq \Set{a+b; a\in r, b \in s}$.
	The \emph{graded Kronecker product} $U \otimes V$ of two graded matrices
	$U \in \F^{\bm{m} \times \bm{n}}$ and $V \in \F^{\bm{p} \times \bm{q}}$
	is the graded $(\bm{m}*\bm{p}) \times (\bm{n}*\bm{q})$-matrix with entries
	\[
		\u{U \otimes V} \coloneqq \begin{psmallmatrix}
			u_{11} V & \cdots & u_{1l} V \\
			\vdots   &        & \vdots   \\
			u_{k1} V & \cdots & u_{kl} V
		\end{psmallmatrix}.
	\]
\end{definition}

If $M$ and $N$ are both flat, then $M \otimes N$ is flat of graded rank $\rk M * \rk N$,
and if the graded matrices $U$ and $V$ represent two morphisms $u$ and $v$ of flat modules,
then $U \otimes V$ represents the morphism $u \otimes v$.
Therefore, the morphisms $\kappa_{ij}$ are represented by the graded matrices $K_{ij} \coloneqq K_i \otimes E_{\rk F_j}$,
where $E_{\rk F}$ denotes the graded $(\rk F)\times(\rk F)$-unit matrix.
\zcref[S]{thm:flat-injective presentation module} now directly yields:

\begin{proposition}
	\label{thm:flat-injective-pres-matrix}
	Let graded matrices $D_1,\dotsc,D_n$ represent a free resolution of a finite dimensional module $M$.
	With $S_{ij}$ and $K_{ij}$ as above, a graded matrix $\Phi$ that represents a flat-injective presentation of $M$
	is given by
	\begin{equation}
		\label{eq:flat-injective-pres-matrix}
		\Phi = S_{0,n-1} K_{1,n-1} \dotsm S_{n-1,0} K_{n0}.
	\end{equation}
\end{proposition}

\paragraph{Algorithmic aspects}
\zcref[S]{thm:flat-injective-pres-matrix} makes it clear that to compute the flat-injective presentation matrix $\Phi$,
it is not necessary to compute the contraction matrices $S^Q_j$ for all $Q$ and $j$.
From \zcref{rmk:representative-index-Q}, with $k_Q = \min Q$, we see that it suffices to compute the matrices
$\bar{S}^{\{k\}}_j$ for all $1 \leq k \leq n-j \leq n$ and define $S_{ij}$ as the block matrix
\[
	S_{ij} = \smashoperator{\bigoplus_{Q \in \binom{n}{n-i}}} p_Q \bigl(r_{\{j_Q\}} (\bar{S}^{\{i_Q\}}_j)\bigr)
\]
for $j = 0,\dotsc,n-1$ and $i = n-1-j$.

\begin{theorem}[Complexity]
	\label{thm:complexity}
	Let $M \in \Pers{\Z^n}$ be finite dimensional module and $F_\bullet$ be a free resolution of $M$ of length $n$.
	Treating the number of parameters $n$ as a constant,
	a flat-injective presentation matrix $\Phi$ of $M$ can be computed in time $\calO(\ell^3)$,
	where $\ell \coloneqq \max_i \abs{\rk F_i}$.
\end{theorem}
\begin{proof}
	A matrix $U$ is \emph{reduced} if the \emph{pivots}
	$\piv U_{*j} \coloneqq \max\Set{j; U_{ij} \neq 0}$ of its nonzero columns are pairwise distinct.
	Assuming that the rows and columns of $\bar{D}_d \coloneqq p_{\{k\}}(D_{d})$ are ordered nondescendingly by grade for each $d$,
	it follows from the proof of \zcref{rmk:obtaining-contraction} that $\bar{S}^{\{k\}}_j$ can be obtained recursively as follows.
	Consider the block matrix $(D_{1}\  E_{p_{\{k\}}(\rk F_0)})$, and let \[U^{(0)} = \begin{psmallmatrix}
		U^{(0)}_{11} & U^{(0)}_{12} \\
		0            & U^{(0)}_{22}
	\end{psmallmatrix}\] 
	be a graded invertible upper triangular block matrix such that $(D_{1}\ E_{p_{\{k\}}(\rk F_0)})\, U^{(0)}$ is reduced;
	such a matrix $U^{(0)}$ can be obtained through a simple column reduction scheme in time $\calO(\ell^3)$.
	Let $\bar{S}^{\{k\}}_0 \coloneqq U^{(0)}_{12}$.
	For $j > 0$, let \[U^{(j)} = \begin{psmallmatrix}
		U^{(j)}_{11} & U^{(j)}_{12} & U^{(j)}_{13} \\
		0            & U^{(j)}_{22} & U^{(j)}_{23} \\
		0            & 0            & U^{(j)}_{33}
	\end{psmallmatrix}\]
	such that $(D_{j+1}\ S_{j-1}\ E_{p_{\{k\}}(\rk  F_j)})\, U_j$ is reduced,
	and set $\bar{S}^{\{k\}}_j \coloneqq U^{(0)}_{13}$.
	Each of the $\frac{1}{2} n(n+1)$ many matrices $S^{\{k\}}_j$ for $1 \leq k \leq n-j \leq n$
	can be computed in time $\calO(\ell^3)$.

	Assembling the matrices $\bar{S}^{\{k\}}_j$ together yields the matrices $S_{ij}$.
	Each matrix $S_{ij}$ and $K_{ij}$ has at most $a\ell$ rows and columns,
	where $a = \binom{n}{\lfloor n/2\rfloor}$.
	Therefore, evaluating the matrix product \eqref{eq:flat-injective-pres-matrix} can be done in time
	$\calO(a^3\ln(2n) \ell^3) = \calO(\ell^3)$.
\end{proof}

\begin{remark}[Obtaining $\bar{s}^Q_\bullet$, revisited]
	We are ultimately interested in forming the product \eqref{eq:t_n0n},
	where $s_{i,j-1} \partial_{i,j} + \partial_{i,j+1} s_{i,j} = \id$ for all $i < n$ and all $j$.
	For $k \geq 0$, let $u_k \coloneqq s \kappa u_{k-1}\colon \Omega_n\otimes F_0 \to \Omega_{n-k}\otimes F_k$ with $u_0 = \id$,
	leaving out indices as in the proof of \zcref{thm:main}.
	In particular, we get $\varphi = u_n$.
	Because $\partial_{ij}$ and $\kappa_{ij}$ represent the boundary morphism of a bounded below double complex,
	one can show inductively that $\partial \kappa u_k = 0$
	and thus $\kappa s u_k = \partial s \kappa u_k$.
	In other words, any collection of contraction morphism $s_{ij}$ has to satisfy
	\[
		(s \kappa u_k)(m) \in \partial^{-1}(\{(\kappa u_k)(m)\})
	\]
	for every $m \in \Omega_n\otimes F_0$.
	This allows to compute the matrix product \eqref{thm:flat-injective-pres-matrix} without computing the entire contraction matrices $S_{ij}$,
	but instead by computing iterated preimages of $\partial_{n-k,k}$.
	Note that this strategy has no better algorithmic complexity,
	but is observed to have a better runtime in practice.
\end{remark}

\subsection{An example}
\label{sec:example}
Consider the $\Z^2$-module $M$ with the following minimal free resolution $F_\bullet$,
where the numbers at the grades correspond to the rows and column indices of the matrices:
\begin{equation*}
	\def\Setup{
		\Axes[->, shorten >=-4pt, shorten <=-4pt](0,0)(3.5,3.5)
		\Grid[0](1){3}[0](1){3}
	}
	\tikzset{
		every module diagram/.append style={x=.4cm,y=.4cm},
		every label/.append style={outer sep=-1pt, label position=left}
	}
	\underbrace{
		\begin{tikzpicture}[module diagram]
			\Setup
			\FreeModule[blue](2,2)[generator,"1"]{}
			\FreeModule[blue](3,3)[generator,"2"]{}
		\end{tikzpicture}
		\xto{\Mtx*[r]{0 & 1 \\ 1 & -1 \\ -1 & 0 \\ 0 & -1}}
		\begin{tikzpicture}[module diagram]
			\Setup
			\FreeModule[blue](0,3)[generator,"1"]{}
			\FreeModule[blue](1,2)[generator,"2"]{}
			\FreeModule[blue](2,1)[generator,"3"]{}
			\FreeModule[blue](3,0)[generator,"4"]{}
		\end{tikzpicture}
		\xto{\Mtx*[r]{1 & 1 & 1 & 0 \\ 0 & -1 & -1 & 1}}
		\begin{tikzpicture}[module diagram]
			\Setup
			\FreeModule[blue](0,1)[generator,"1"]{}
			\FreeModule[blue](1,0)[generator,"2"]{}
		\end{tikzpicture}
	}_{F_\bullet}
	\longrightarrow
	\underbrace{
		\begin{tikzpicture}[module diagram]
			\Setup
			\begin{scope}[blue]
				\filldraw[fill opacity=0.2] (.5,-.5) -| (2.5,2.5) -| (-.5,.5) -| cycle;
				\filldraw[fill opacity=0.2] (.5,.5) rectangle (1.5,1.5);
				\draw[->] (1,.25) to["$\Mtx{1 \\ 0}$" {right, at start, anchor=north west, inner sep=1pt}] (1,.75);
				\draw[->] (.25,1) to["$\Mtx{0 \\ 1}$" {above, at start, anchor=-55, inner sep=1pt}] (.75,1);
				\draw[->] (1.25,1.25) to["$\scriptstyle (1\:1)$" {left, at end, anchor=south, inner sep=1pt}] (1.75,1.75);
			\end{scope}
		\end{tikzpicture}
	}_{M}
\end{equation*}
Passing to the chain complex $\colim_{\{2\}} F_\bullet$
amounts to forgetting the second (vertical) coordinate of the grading.
This gives the acyclic chain complex of $\Z$-modules given by the solid arrows of the following diagram:
\begin{equation}
	\label{eq:ex:colim-2}
	\colim_{\{2\}} F_\bullet\colon
	\raisebox{0pt}[\height][.6\depth]{
	\begin{tikzcd}[ampersand replacement=\&, column sep=4em, row sep=3em, baseline=(b.base)]
		|[alias=b]| 0 \rar \&[-3em]
		\begin{tikzpicture}[module diagram]
			\draw[|->] (0,0) -- (6.5,0);
			\draw[blue, yshift=6pt, thick] (4,0) node[generator,"1" above]{} -- (6.5,0);
			\draw[blue, yshift=4pt, thick] (6,0) node[generator,"2" {above, yshift=2pt}]{} -- (6.5,0);
		\end{tikzpicture}
		\rar{\Mtx*[r]{0 & 1 \\ 1 & -1 \\ -1 & 0 \\ 0 & -1}}
		\ar[from=r, bend left=2em, dashed, "{\bar{S}^{\{2\}}_1 = \Mtx*[r]{0 & 0 & -1 & 0 \\ 0 & 0 & 0 & -1}}"]
		\&
		\begin{tikzpicture}[module diagram]
			\draw[|->] (0,0) -- (6.5,0);
			\draw[blue, yshift=10pt, thick] (0,0) node[generator,"1" above]{} -- (6.5,0);
			\draw[blue, yshift= 8pt, thick] (2,0) node[generator,"2" {above, yshift=2pt}]{} -- (6.5,0);
			\draw[blue, yshift= 6pt, thick] (4,0) node[generator,"3" {above, yshift=4pt}]{} -- (6.5,0);
			\draw[blue, yshift= 4pt, thick] (6,0) node[generator,"4" {above, yshift=6pt}]{} -- (6.5,0);
		\end{tikzpicture}
		\rar{\Mtx*[r]{1 & 1 & 1 & 0 \\ 0 & -1 & -1 & 1}}
		\ar[from=r, bend left=2em, dashed, "{\bar{S}^{\{2\}}_0 \Mtx*[r]{1 & 1 \\ 0 & -1 \\ 0 & 0 \\ 0 & 0}}"]
		\&
		\begin{tikzpicture}[module diagram]
			\draw[|->] (0,0) -- (6.5,0);
			\draw[blue, yshift=6pt, thick] (0,0) node[generator,"1" above]{} -- (6.5,0);
			\draw[blue, yshift=4pt, thick] (2,0) node[generator,"2" {above, yshift=2pt}]{} -- (6.5,0);
		\end{tikzpicture}
		\rar \&[-3em] 0
	\end{tikzcd}
	}
\end{equation}
It is contractible per the chain contraction $\bar{s}^{\{2\}}_\bullet$ given by the dashed morphisms.
Analogously, forgetting the first (horizontal) coordinate, we obtain the complex
\begin{equation}
	\label{eq:ex:colim-1}
	\tikzset{every module diagram/.append style={y={(0,0.75)}}}
	\colim_{\{1\}} F_\bullet\colon
	\begin{tikzcd}[ampersand replacement=\&, column sep=5em, baseline=(b.base)]
		|[alias=b]| 0 \rar \&[-4em]
		\begin{tikzpicture}[module diagram]
			\draw[|->] (0,0) -- (0, 6.5);
			\draw[blue, xshift=6pt, thick] (0,4) node[generator,"1" right]{} -- (0,6.5);
			\draw[blue, xshift=4pt, thick] (0,6) node[generator,"2" {right, xshift=2pt}]{} -- (0,6.5);
		\end{tikzpicture}
		\rar{\Mtx*[r]{0 & 1 \\ 1 & -1 \\ -1 & 0 \\ 0 & -1}}
		\ar[from=r, dashed, bend left, "{\bar{S}^{\{1\}}_1 = \Mtx*[r]{1 & 1 & 0 & 0 \\ 1 & 0 & 0 & 0}}"]
		\&
		\begin{tikzpicture}[module diagram]
			\draw[|->] (0,0) -- (0,6.5);
			\draw[blue, xshift=10pt, thick] (0,0) node[generator,"4" right]{} -- (0,6.5);
			\draw[blue, xshift= 8pt, thick] (0,2) node[generator,"3" {right, xshift=2pt}]{} -- (0,6.5);
			\draw[blue, xshift= 6pt, thick] (0,4) node[generator,"2" {right, xshift=4pt}]{} -- (0,6.5);
			\draw[blue, xshift= 4pt, thick] (0,6) node[generator,"1" {right, xshift=6pt}]{} -- (0,6.5);
		\end{tikzpicture}
		\rar{\Mtx*[r]{1 & 1 & 1 & 0 \\ 0 & -1 & -1 & 1}}
		\ar[from=r, dashed, bend left, "{\bar{S}^{\{1\}}_0 = \Mtx*[r]{0 & 0 \\ 0 & 0 \\ 1 & 0 \\ 1 & 1}}"]
		\&
		\begin{tikzpicture}[module diagram]
			\draw[|->] (0,0) -- (0,6.5);
			\draw[blue, xshift=6pt, thick] (0,0) node[generator,"2" right]{} -- (0,6.5);
			\draw[blue, xshift=4pt, thick] (0,2) node[generator,"1" {right, xshift=2pt}]{} -- (0,6.5);
		\end{tikzpicture}
		\rar \&[-4em]
		0
	\end{tikzcd}
\end{equation}
together with the contraction $\bar{s}^{\{1\}}_\bullet$ represented by matrices $\bar{S}^{\{1\}}_\bullet$ (dashed arrows).
Then $s_{1\bullet}\coloneqq s^{\{1\}}_\bullet \oplus s^{\{2\}}_\bullet$
is a chain contraction of $\Omega_1 F_\bullet = \Colim_{\{1\}} F_\bullet \oplus \Colim_{\{2\}} F_\bullet$.
By \zcref{rmk:representative-index-Q}, we set $i_{\{1,2\}} = 2$ and use the entries of $\bar{S}^{\{2\}}_\bullet$ for $\bar{S}^{\{1,2\}}_\bullet$.
Then $s_{0\bullet} \coloneqq \Colim_{\{1,2\}} s^{\{2\}}_\bullet$ is a chain contraction of $\Omega_0 F_\bullet = \Colim_{\{1,2\}} F_\bullet$.
We obtain the flat-injective presentation matrix
\begin{multline*}
	\Phi = S_{01} K_{11} S_{10} K_{20}
	= S^{\{1,2\}}_1
	  (E \ {-E})
	  \begin{pmatrix} S^{\{1\}}_0 & 0\\ 0& S^{\{2\}}_2\end{pmatrix}
	  \begin{pmatrix} E \\ E\end{pmatrix}\\
	\begin{aligned}[t] &=
		\mleft(\begin{smallarray}{rrrr}0 & 0 & -1 & 0 \\ 0 & 0 & 0 & -1\end{smallarray}\mright)
		\mleft(\begin{smallarray}{rrrr|rrrr}1 & 0 & 0 & 0 & -1 & 0 & 0 & 0 \\ 0 & 1 & 0 & 0 & 0 & -1 & 0 & 0 \\ 0 & 0 & 1 & 0 & 0 & 0 & -1 & 0 \\ 0 & 0 & 0 & 1 & 0 & 0 & 0 & -1\end{smallarray}\mright)
		\smash{\mleft(\begin{smallarray}{rr|rr} 0 & 0 \\ 0 & 0 \\ 1 & 0 \\ 1 & 1 \\\hline & & 1 & 1 \\ & & 0 & -1 \\ & & 0 & 0 \\ & & 0 & 0\end{smallarray}\mright)}
		\mleft(\begin{smallarray}{cc}1 & 0 \\ 0 & 1 \\ \hline 1 & 0 \\ 0 & 1\end{smallarray}\mright)\\
	&= \Mtx*[r]{-1 & 0 \\ -1 & -1}.\end{aligned}
\end{multline*}
With $\cg^\Phi = \rk F_0 = \Mtx{(0, 1)\\(1, 0)}$ and $\rg^\Phi = \rk F_2\Shift{\one} = \Mtx{(1, 1) \\ (2,2)}$,
one checks that $\Phi$ is anti-valid and thus represents a morphism $\varphi\colon F_0 \to \nu F_2\Shift{\one}$.
We checked in \zcref{ex:resolutions:ctd} already that this is morphism is a flat-injective presentation of $M$.
By changing the basis of the codomain, one obtains the following more symmetric flat-injective presentation matrix of $M$:
\begin{equation*}
		\begin{pNiceMatrix}[first-col,first-row,small]
			& (0,1) & (1,0) \\
			(1,1) & 1 & -1 \\
			(2,2) & 1 & \phantom{-}1
		\end{pNiceMatrix}.
\end{equation*}

\subsection{Implementation}
\label{sec:implementation}
An implementation of the algorithm described above is available in the software package \texttt{FlangePresentations.jl} \cite{FlangePresentations.jl}.
The software provides methods for loading a free resolution in the scc2020 format 
and thus provides interoperability with the software packages \texttt{mpfree}~\cite{Kerber:2021} and \texttt{2pac}~\cite{Lenzen:2023}
for computing minimal free resolutions of persistent homology.
Experiments show that for realistic sizes of resolutions of persistent homology,
computing the flat-injective matrix with our software is alomost instantaneous.
The above example is available as \texttt{example.scc2020} in the repository.
For more details on using the software, we refer to the readme file in the repository.

\section{Flat-injective presentations of persistent homology}
We finish this paper by some last remarks on computing a flat-injective presentation of the homology of a free chain complex.
Let $(C_\bullet, \partial_\bullet)$ be an eventually acyclic chain complex of free modules in $\Pers{n}$.
Of course, one can compute a minimal free resolution of $H_d(C_\bullet)$, and then use \zcref{thm:flat-injective presentation module}
to obtain a minimal flat-injective presentation of $H_d(C_\bullet)$.
We outline a slightly different approach now.

By \zcref{thm:main}, we have a quasi-isomorphism $\phi\colon C_\bullet \to \nu C_\bullet [n]$,
where $\nu C_\bullet[n]$ is a complex of injective modules.
Let $f\colon F \onto \ker\partial_d$ be a free cover and $i\colon \coker \nu\partial_{d+n+1} \into I$ be an injective hull.

\begin{theorem}
	The composition
	\begin{equation}
		\label{eq:flange-pres homology}
		F \xto{f} C_d \xto{\phi_d} \nu C_{d+n} \xto{i} I
	\end{equation}
	is a flat-injective presentation of $H_d(C_\bullet)$.
	It is minimal if and only $F$ and $I$ are a minimal free cover and minimal injective hull, respectively.
\end{theorem}

Note that $\coker \nu\partial_{d+n+1} \cong (\ker \partial_{d+n+1}^\dagger)^*$
is the Matlis dual of a kernel of a morphism of free modules,
and if $i^*\colon I^* \to C_{d+n}^\dagger$ is a (minimal) free cover of $\ker \partial_{d+n+1}^\dagger$,
then $i$ is a (minimal) injective hull of $\coker \nu\partial_{d+n+1}$.
Therefore, a matrix representing \eqref{eq:flange-pres homology} can be computed
using the method outlined in this paper, given that an algorithm to compute (minimal free covers of) kernels of morphisms of free modules is available.

For more than two parameters, it might be an advantage for the computation of flat-injective presentation matrices
to compute just the minimal covers of $\ker \partial_d$ and $\ker \partial_{d+n+1}^\dagger$,
instead of computing the entire free resolution of $H_d(C_\bullet)$.
For two parameters, these two minimal covers already determine the entire minimal free resolution, so there is no efficiency gain expected.

\printbibliography
\end{document}